\documentclass[american,english]{article}
\usepackage[latin9]{inputenc}
\usepackage{textcomp}
\usepackage{amsmath}
\usepackage{amsthm}
\usepackage{amssymb}
\usepackage{wasysym}

\makeatletter
\numberwithin{equation}{section}
\numberwithin{figure}{section}
\theoremstyle{plain}
\newtheorem{thm}{\protect\theoremname}[section]
  \theoremstyle{plain}
  \newtheorem{cor}[thm]{\protect\corollaryname}
  \theoremstyle{plain}
  \newtheorem{prop}[thm]{\protect\propositionname}
  \theoremstyle{definition}
  \newtheorem{defn}[thm]{\protect\definitionname}
  \theoremstyle{plain}
  \newtheorem{lem}[thm]{\protect\lemmaname}
  \theoremstyle{remark}
  \newtheorem{rem}[thm]{\protect\remarkname}
  \theoremstyle{remark}
  \newtheorem{claim}[thm]{\protect\claimname}
  \theoremstyle{definition}
  \newtheorem{problem}[thm]{\protect\problemname}

\makeatother

\usepackage{babel}
  \addto\captionsamerican{\renewcommand{\claimname}{Claim}}
  \addto\captionsamerican{\renewcommand{\corollaryname}{Corollary}}
  \addto\captionsamerican{\renewcommand{\definitionname}{Definition}}
  \addto\captionsamerican{\renewcommand{\lemmaname}{Lemma}}
  \addto\captionsamerican{\renewcommand{\problemname}{Problem}}
  \addto\captionsamerican{\renewcommand{\propositionname}{Proposition}}
  \addto\captionsamerican{\renewcommand{\remarkname}{Remark}}
  \addto\captionsamerican{\renewcommand{\theoremname}{Theorem}}
  \addto\captionsenglish{\renewcommand{\claimname}{Claim}}
  \addto\captionsenglish{\renewcommand{\corollaryname}{Corollary}}
  \addto\captionsenglish{\renewcommand{\definitionname}{Definition}}
  \addto\captionsenglish{\renewcommand{\lemmaname}{Lemma}}
  \addto\captionsenglish{\renewcommand{\problemname}{Problem}}
  \addto\captionsenglish{\renewcommand{\propositionname}{Proposition}}
  \addto\captionsenglish{\renewcommand{\remarkname}{Remark}}
  \addto\captionsenglish{\renewcommand{\theoremname}{Theorem}}
  \providecommand{\claimname}{Claim}
  \providecommand{\corollaryname}{Corollary}
  \providecommand{\definitionname}{Definition}
  \providecommand{\lemmaname}{Lemma}
  \providecommand{\problemname}{Problem}
  \providecommand{\propositionname}{Proposition}
  \providecommand{\remarkname}{Remark}
\providecommand{\theoremname}{Theorem}

\begin{document}

\title{The Congruence Subgroup Problem \\
for the Free Metabelian group on $n\geq4$ generators}

\author{David El-Chai Ben-Ezra}
\maketitle
\begin{abstract}
The congruence subgroup problem for a finitely generated group $\Gamma$
asks whether the map $\widehat{Aut\left(\Gamma\right)}\to Aut(\hat{\Gamma})$
is injective, or more generally, what is its kernel $C\left(\Gamma\right)$?
Here $\hat{X}$ denotes the profinite completion of $X$. It is well
known that for finitely generated free abelian groups $C\left(\mathbb{Z}^{n}\right)=\left\{ 1\right\} $
for every $n\geq3$, but $C\left(\mathbb{Z}^{2}\right)=\hat{F}_{\omega}$,
where $\hat{F}_{\omega}$ is the free profinite group on countably
many generators. 

Considering $\Phi_{n}$, the free metabelian group on $n$ generators,
it was also proven that $C\left(\Phi_{2}\right)=\hat{F}_{\omega}$
and $C\left(\Phi_{3}\right)\supseteq\hat{F}_{\omega}$. In this paper
we prove that $C\left(\Phi_{n}\right)$ for $n\geq4$ is abelian.
So, while the dichotomy in the abelian case is between $n=2$ and
$n\geq3$, in the metabelian case it is between $n=2,3$ and $n\geq4$.
\end{abstract}
\textbf{}

\tableofcontents{}

\section{Introduction}

The classical congruence subgroup problem (CSP) asks for, say, $G=SL_{n}\left(\mathbb{Z}\right)$
or $G=GL_{n}\left(\mathbb{Z}\right)$, whether every finite index
subgroup of $G$ contains a principal congruence subgroup, i.e. a
subgroup of the form $G\left(m\right)=\ker\left(G\to GL_{n}\left(\mathbb{Z}/m\mathbb{Z}\right)\right)$
for some $0\neq m\in\mathbb{Z}$. Equivalently, it asks whether the
natural map $\hat{G}\to GL_{n}(\hat{\mathbb{Z}})$ is injective, where
$\hat{G}$ and $\hat{\mathbb{Z}}$ are the profinite completions of
the group $G$ and the ring $\mathbb{Z}$, respectively. More generally,
the CSP asks what is the kernel of this map. It is a classical $19^{\underline{th}}$
century result that the answer is negative for $n=2$. Moreover (but
not so classical, cf. \cite{key-17}, \cite{key-4}), the kernel in
this case is $\hat{F}_{\omega}$ - the free profinite group on a countable
number of generators. On the other hand, it was proved in the sixties
by Mennicke \cite{key-22} and Bass-Lazard-Serre \cite{key-23} that
for $n\geq3$ the answer is affirmative, and the kernel is therefore
trivial.

By the observation $GL_{n}\left(\mathbb{Z}\right)\cong Aut\left(\mathbb{Z}^{n}\right)=Out\left(\mathbb{Z}^{n}\right)$,
the CSP can be generalized as follows: Let $\Gamma$ be a group and
$G\leq Aut\left(\Gamma\right)$ (resp. $G\leq Out\left(\Gamma\right)$).
For a finite index characteristic subgroup $M\leq\Gamma$ denote 
\begin{eqnarray*}
G\left(M\right) & = & \ker\left(G\to Aut\left(\Gamma/M\right)\right)\\
(\textrm{resp.}\,\,\,G\left(M\right) & = & \ker\left(G\to Out\left(\Gamma/M\right)\right)).
\end{eqnarray*}
Such a $G\left(M\right)$ will be called a ``principal congruence
subgroup'' and a finite index subgroup of $G$ which contains $G\left(M\right)$
for some $M$ will be called a ``congruence subgroup''. The CSP
for the pair $\left(G,\Gamma\right)$ asks whether every finite index
subgroup of $G$ is a congruence subgroup. In some sense, the CSP
tries to understand whether every finite quotient of $G$ comes from
a finite quotient of $\Gamma$. 

One can easily see that the CSP is equivalent to the question: Is
the congruence map $\hat{G}=\underleftarrow{\lim}G/U\to\underleftarrow{\lim}G/G\left(M\right)$
injective? Here, $U$ ranges over all finite index normal subgroups
of $G$, and $M$ ranges over all finite index characteristic subgroups
of $\Gamma$. When $\Gamma$ is finitely generated, it has only finitely
many subgroups of given index $m$, and thus, the charateristic subgroups
$M_{m}=\cap\left\{ \Delta\leq\Gamma\,|\,\left[\Gamma:\Delta\right]=m\right\} $
are of finite index in $\Gamma$. Hence, one can write $\hat{\Gamma}=\underleftarrow{\lim}_{m\in\mathbb{N}}\Gamma/M_{m}$
and have\footnote{By the celebrated theorem of Nikolov and Segal which asserts that
every finite index subgroup of a finitely generated profinite group
is open \cite{key-17-1}, the second inequality is actually an equality.
However, we do not need it. } 
\begin{eqnarray*}
\underleftarrow{\lim}G/G\left(M\right) & = & \underleftarrow{\lim}_{m\in\mathbb{N}}G/G\left(M_{m}\right)\leq\underleftarrow{\lim}_{m\in\mathbb{N}}Aut(\Gamma/M_{m})\\
 & \leq & Aut(\underleftarrow{\lim}_{m\in\mathbb{N}}(\Gamma/M_{m}))=Aut(\hat{\Gamma})\,\,\,\,(\textrm{resp.}\,\,Out(\hat{\Gamma})).
\end{eqnarray*}
Therefore, when $\Gamma$ is finitely generated, the CSP is equivalent
to the question: Is the congruence map $\hat{G}\to Aut(\hat{\Gamma})$
(resp. $\hat{G}\to Out(\hat{\Gamma})$) injective? More generally,
the CSP asks what is the kernel $C\left(G,\Gamma\right)$ of this
map. For $G=Aut\left(\Gamma\right)$ we will also use the simpler
notation $C\left(\Gamma\right)=C\left(G,\Gamma\right)$. 

The classical congruence subgroup results mentioned above can therefore
be reformulated as $C\left(\mathbb{Z}^{2}\right)=\hat{F}_{\omega}$
while $C\left(\mathbb{Z}^{n}\right)=\left\{ e\right\} $ for $n\geq3$.
So the finite quotients of $GL_{n}\left(\mathbb{Z}\right)$ are closely
related to the finite quotients of $\mathbb{Z}^{n}$ when $n\geq3$,
but the finite quotients of $GL_{2}\left(\mathbb{Z}\right)$ are far
from being understandable by the finite quotients of $\mathbb{Z}^{2}$.

Very few results are known when $\Gamma$ is non-abelian. Most of
the results are related to $\Gamma=\pi_{g,n}$, the fundamental group
of $S_{g,n}$, the closed surface of genus $g$ with $n$ punctures.
In these cases one can take $G=PMod\left(S_{g,n}\right)$, the pure
mapping class group of $S_{g,n}$, and can naturally view it as a
subgroup of $Out\left(\pi_{g,n}\right)$ (cf. \cite{key-20}, chapter
8). Considering these cases, it is known that:
\begin{thm}
\label{thm:MCG}For $g=0,1,2$ and every $n\geq0,1,0$ respectively,
we have $C\left(PMod\left(S_{g,n}\right),\pi_{g,n}\right)=\left\{ 1\right\} $.
\end{thm}

Note that when $g=1$ and $n=0$, $\pi_{1,0}\cong\mathbb{Z}^{2}$
and $PMod\left(S_{1,0}\right)\cong SL_{2}\left(\mathbb{Z}\right)$,
so $C\left(PMod\left(S_{1,0}\right),\pi_{1,0}\right)=C\left(SL_{2}\left(\mathbb{Z}\right),\mathbb{Z}^{2}\right)=\hat{F}_{\omega}$.
The cases for $g=0$ were proved in \cite{key-16-1} (see also \cite{key-18}),
the cases for $g=1$ were proved in \cite{key-3} (see also \cite{key-19},
\cite{key-5}), and the cases for $g=2$ were proved in \cite{key-19}
(see also \cite{key-6-1} for the specific case where $g=2$ and $n=0$).
In particular, as $PMod\left(S_{1,1}\right)$ is isomorphic to the
special outer-automorphism group of $F_{2}$, we have an affirmative
answer for the full outer-automorphism group of $F_{2}$, and by some
standard arguments it shows that actually $C\left(F_{2}\right)$ is
trivial (see \cite{key-5}, \cite{key-7}). Note that for every $n>0$,
$\pi_{g,n}\cong F_{2g+n-1}$ = the free group on $2g+n-1$ generators.
Hence, the above solved cases give an affirmative answer for various
subgroups of the outer-automorphism group of finitely generated free
groups, while the CSP for the full $Aut\left(F_{d}\right)$ when $d\geq3$
is still unsettled, and so is the situation with $PMod\left(S_{g,n}\right)$
when $g\geq3$. 

All the above settled cases have a common property which plays a crucial
role in the proof of Theorem \ref{thm:MCG}: There is an intrinsic
description of $G$ by iterative extension process by virtually free
groups (groups which have a finite index free subgroup). Actually,
in these cases, in some sense, we do understand the finite quotients
of $G$, and the CSP tells us that these quotients are closely related
to the finite quotients of $\Gamma$. This situation changes when
we pass to $G=Aut\left(F_{d}\right)$ for $d\geq3$ or $PMod\left(S_{g,n}\right)$
for $g\geq3$. In these cases we do not have a description of $G$
that can help to understand the finite quotients of $G$. So in some
sense, all the known cases do not give us a new understanding of the
finite quotients of $G$. Considering the abelian case, what makes
the result of Mennicke and Bass-Lazard-Serre so special is that it
not only shows that the finite quotients of $GL_{n}\left(\mathbb{Z}\right)$
are related to the finite quotients of $\mathbb{Z}^{n}$, but also
gives us a description of the finite quotients of $GL_{n}\left(\mathbb{Z}\right)$,
which we have not known without this result.

Denote now the free metabelian group on $n$ generators by $\Phi_{n}=F_{n}/F_{n}''$.
Considering the metabelian case, it was shown in \cite{key-7} (see
also \cite{key-6}) that $C\left(\Phi_{2}\right)=\hat{F}_{\omega}$.
In addition, it was proven there that $C\left(\Phi_{3}\right)\supseteq\hat{F}_{\omega}$.
So, the finite quotients of $Aut\left(\Phi_{2}\right)$ and $Aut\left(\Phi_{3}\right)$
are far from being connected to the finite quotients of $\Phi_{2}$
and $\Phi_{3}$, respectively. 

Here comes the main theorem of this paper: 
\begin{thm}
\label{thm:main}For every $n\geq4$, $C\left(IA\left(\Phi_{n}\right),\Phi_{n}\right)$
is central in $\widehat{IA\left(\Phi_{n}\right)}$, where 
\[
IA\left(\Phi_{n}\right)=\ker\left(Aut\left(\Phi_{n}\right)\to Aut\left(\Phi_{n}/\Phi'_{n}\right)=GL_{n}\left(\mathbb{Z}\right)\right).
\]
\end{thm}

Using the commutative exact diagram (see $\varoint$\ref{sec:Inferences})
\[
\begin{array}{ccccccc}
\widehat{IA\left(\Phi_{n}\right)} & \to & \widehat{Aut\left(\Phi_{n}\right)} & \to & \widehat{GL_{n}\left(\mathbb{Z}\right)} & \to & 1\\
 & \searrow & \downarrow &  & \downarrow\\
 &  & Aut(\hat{\Phi}_{n}) & \to & GL_{n}(\hat{\mathbb{Z}})
\end{array}
\]
and the fact that $\widehat{GL_{n}\left(\mathbb{Z}\right)}\to GL_{n}(\hat{\mathbb{Z}})$
is injective for $n\geq3$, we obtain that $C\left(IA\left(\Phi_{n}\right),\Phi_{n}\right)$
is mapped onto $C\left(\Phi_{n}\right)$. Therefore we deduce that: 
\begin{thm}
\label{thm:full}For every $n\geq4$, $C\left(\Phi_{n}\right)$ is
abelian.
\end{thm}

This is dramatically different from the cases of $n=2,3$ described
above. Theorem \ref{thm:full} tells us that when $n\geq4$ the situation
changes, and the finite quotients of $Aut\left(\Phi_{n}\right)$ are
closely related to the finite quotients of $\Phi_{n}$ in the following
manner:
\begin{cor}
\label{cor:description}Let $n\geq4$. Then, for every finite index
subgroup $H\leq G=Aut\left(\Phi_{n}\right)$, there exists a finite
index characteristic subgroup $M\leq\Phi_{n}$ and $r\in\mathbb{N}$
such that $G\left(M\right)'G\left(M\right)^{r}\subseteq H$.
\end{cor}

Note that by a theorem of Bachmuth and Mochizuki \cite{key-24}, $Aut\left(F_{n}\right)\to Aut\left(\Phi_{n}\right)$
is surjective for every $n\geq4$, and thus $G=Aut\left(\Phi_{n}\right)$
is finitely generated. Hence, the principal congruence subgroups of
the form $G\left(M\right)$ are finitely generated, and thus, the
subgroups of the form $G\left(M\right)'G\left(M\right)^{r}$ are also
of finite index in $Aut\left(\Phi_{n}\right)$. Therefore, the quotients
of the form $Aut\left(\Phi_{n}\right)/G\left(M\right)'G\left(M\right)^{r}$
describe all the finite quotients of $Aut\left(\Phi_{n}\right)$.
In particular, our theorem gives us a description of the finite quotients
of $Aut\left(\Phi_{n}\right)$ when $n\geq4$ - just like the theorem
of \cite{key-22} and \cite{key-23} gives for $GL_{n}\left(\mathbb{Z}\right)$
when $n\geq3$. Corollary \ref{cor:description} obviously does not
hold for $n=2,3$. So, the picture is that while the dichotomy in
the abelian case is between $n=2$ and $n\geq3$, in the metabelian
case we have a dichotomy between $n=2,3$ and $n\geq4$. 

In \cite{key-24-1}, Kassabov and Nikolov showed that $\ker(\widehat{SL_{n}\left(\mathbb{Z}\left[x\right]\right)}\to SL_{n}(\widehat{\mathbb{Z}\left[x\right]}))$
is central and not finitely generated, when $n\geq3$. In \cite{key-14}
we use their techniques and an interesting surjective representation
\[
IA\left(\Phi_{n}\right)\twoheadrightarrow\ker(GL_{n-1}\left(\mathbb{Z}[x^{\pm1}]\right)\overset{x\mapsto1}{\longrightarrow}GL_{n-1}\left(\mathbb{Z}\right))
\]
to show also that:
\begin{thm}
\label{cor:not finitely}For every $n\geq4$, $C\left(IA\left(\Phi_{n}\right),\Phi_{n}\right)$
is not finitely generated.
\end{thm}

We remark that despite the result of the latter theorem, we do not
know whether $C\left(\Phi_{n}\right)$ is also not finitely generated.
In fact we cannot even prove at this point that it is not trivial
(for more, see $\varoint$\ref{sec:Inferences}). 

We would like now to give a conceptual explanation for the dichotomy
between $n=2,3$ and $n\geq4$. Let us recall first the strategy of
Bachmuth and Mochizuki \cite{key-24}, showing that the natural map
$Aut\left(F_{n}\right)\to Aut\left(\Phi_{n}\right)$ is surjective
for $n\geq4$. They start with the observation that as $Aut\left(F_{n}\right)\to GL_{n}(\mathbb{Z})$
is surjective, it is enough to show that every element of $IA\left(\Phi_{n}\right)$
is induced by an element of $IA\left(F_{n}\right)=\ker\left(Aut\left(F_{n}\right)\to Aut\left(F_{n}/F'_{n}\right)=GL_{n}\left(\mathbb{Z}\right)\right)$.
From here, the basic background for their proof are the following
facts:
\begin{itemize}
\item For every $n$ (not only for $n\geq4$), $IA\left(\Phi_{n}\right)$
can naturally be viewed as a subgroup of $GL_{n}\left(R_{n}\right)$
where $R_{n}=\mathbb{Z}[x_{1}^{\pm1},\ldots,x_{n}^{\pm1}]$ is the
free Laurent polynomial ring on $n$ commutative variables over $\mathbb{Z}$
(see $\varoint$\ref{sec:structure} for the details). It can be easily
shown that for every $1\leq i\leq n$, this embedding contains a copy
of the group
\[
IGL_{n-1,i}:=\ker(GL_{n-1}\left(R_{n}\right)\overset{x_{i}\mapsto1}{\longrightarrow}GL_{n-1}\left(R_{n}/(x_{i}-1)R_{n}\right)).
\]
\item By a classical result of Magnus (\cite{key-1}, Chapter 3, Theorem
N4) the group $IA\left(F_{n}\right)$ is finitely generated by a well
described generating set of its elements, say $S$ (also here, it
is true for every $n$). Denote the image of $S$ in $IA\left(\Phi_{n}\right)$
by $\bar{S}$. 
\end{itemize}
The technique of \cite{key-24} was to show that when $n\geq4$ the
set $\bar{S}$ generates the whole of $IA\left(\Phi_{n}\right)$.
In Section 3, \cite{key-24} start with presenting a variety of types
of elements that are contained in the subgroup of $IA\left(\Phi_{n}\right)$
generated by $\bar{S}$ - here, \cite{key-24} already needed the
assumption of $n\geq4$. In section 5, \cite{key-24} show that every
element $g\in IA\left(\Phi_{n}\right)$ can be written as a product
of elements
\[
g=h_{0}k_{1}h_{1}k_{2}h_{2}\cdot\ldots\cdot k_{n}h_{n}
\]
where $k_{i}\in IGL_{n-1,i}$ and $h_{i}$ are elements generated
by $\bar{S}$ (by Section 3) - so it remains to show that $IGL_{n-1,i}$
is generated by $\bar{S}$. Then, in the beginning of Section 4, by
some classical results from Algebraic K-Theory, \cite{key-24} manage
to give a description for a generating set to $IGL_{n-1,i}$. From
here, the rest of Section 4 is devoted to show that the generating
set of $IGL_{n-1,i}$ can be built from the elements of $\bar{S}$. 

The aforementioned results from Algebraic K-Theory are strongly leaning
on the assumption $n\geq4$, i.e. $n-1\geq3$. Actually, when $n=3$,
i.e. $n-1=2$, the situation is completely different, and leads to
the fact that $Aut(\Phi_{3})$ is not finitely generated (See \cite{key-2}).
In some sense, what we show in this paper is that this difference
between $Aut(\Phi_{3})$ and $Aut(\Phi_{n\geq4})$, combined with
the dichotomy between $n=2$ and $n\geq3$ in the CSP for the abelian
case, induces a dichotomy between $n=2,3$ and $n\geq4$ in the CSP
for the metabelian case.

The main line of the proof of Theorem \ref{thm:main}, is as follows:
For $G=IA\left(\Phi_{n}\right)$ we first take the principal congruence
subgroups $G\left(M_{n,m}\right)$ where $M_{n,m}=\left(\Phi'_{n}\Phi_{n}^{m}\right)'\left(\Phi'_{n}\Phi_{n}^{m}\right)^{m}$.
By \cite{key-6}, $\hat{\Phi}_{n}=\underleftarrow{\lim}\left(\Phi_{n}/M_{n,m}\right)$,
and thus we deduce that the subgroups of the form $G\left(M_{n,m}\right)$
are enough to represent the congruence subgroups of $IA(\Phi_{n})$
in the sense that every congruence subgroup contains one of these
principal congruence subgroups. Then, we follow the steps of the theorem
of Bachmuth and Mochizuki \cite{key-24}, showing that $Aut\left(F_{n}\right)\to Aut\left(\Phi_{n}\right)$
is surjective for $n\geq4$, and we try to build $G\left(M_{n,m}\right)$
using elements of $\left\langle IA\left(\Phi_{n}\right)^{m}\right\rangle $.

Along this paper, mostly in $\varoint$\ref{sec:elementary} and Claim
\ref{claim:}, we present a variety of types of elements that are
contained in $\left\langle IA\left(\Phi_{n}\right)^{m}\right\rangle $.
In $\varoint$\ref{sec:The-main-lemma} we prove a main lemma, which
can be viewed as a counterpart of Section 5 in \cite{key-24}. A counterpart
of Section 4 in \cite{key-24} is proven in Section 7 of \cite{key-14}
(see Lemma \ref{thm:stage 2} in this paper). These parts are combined
together in $\varoint$\ref{sec:structure-1}, and by some additional
results from algebraic K-theory we get that for every $m$
\[
\left\langle IA\left(\Phi_{n}\right)^{m}\right\rangle G\left(M_{n,m^{4}}\right)/\left\langle IA\left(\Phi_{n}\right)^{m}\right\rangle \,\,\,\,\,\,\,\,\textrm{(notice\,\,the}\,\,m^{4}\textrm{)}
\]
is finite and central in $IA\left(\Phi_{n}\right)/\left\langle IA\left(\Phi_{n}\right)^{m}\right\rangle $.
Hence, $\left\langle IA\left(\Phi_{n}\right)^{m}\right\rangle $ is
of finite index in $IA\left(\Phi_{n}\right)$. In particular, as every
normal subgroup of index $m$ in $IA\left(\Phi_{n}\right)$ contains
$\left\langle IA\left(\Phi_{n}\right)^{m}\right\rangle $, we deduce
that the groups of the form $\left\langle IA\left(\Phi_{n}\right)^{m}\right\rangle $
are enough to represent the finite index subgroups of $IA\left(\Phi_{n}\right)$.
From here, it follows easily that $C\left(IA\left(\Phi_{n}\right),\Phi_{n}\right)$
is central in $\widehat{IA\left(\Phi_{n}\right)}$ (see Corollary
\ref{cor:kernel}). 

We hope that the solution of the free metabelian case will help to
understand some new cases of non-abelian groups, such as the automorphism
group of a free group and the mapping class group of a surface. The
immediate next challenges are the automorphism groups of free solvable
groups.

Let us point out that, as remarked in$\varoint$5 in \cite{key-7},
one can deduce from Theorem \ref{thm:full} that for every $n\geq4$,
$Aut\left(\Phi_{n}\right)$ is not large, i.e does not contain a finite
index subgroup which can be mapped onto a free group. This is in contrast
with $Aut\left(\Phi_{2}\right)$ and $Aut\left(\Phi_{3}\right)$ which
are large.

The paper is organized as follows: In $\varoint$\ref{sec:structure}
we present some notations and discuss $IA\left(\Phi_{n}\right)$ and
some of its subgroups. Then, up to a main lemma, in $\varoint$\ref{sec:structure-1}
we prove the main theorem of the paper, Theorem \ref{thm:main}. In
$\varoint$\ref{sec:elementary} we present some elements of $\left\langle IA\left(\Phi_{n}\right)^{m}\right\rangle $
which we use in the proof of the main lemma. In $\varoint$\ref{sec:The-main-lemma}
we prove the main lemma. We end the paper with the proof of Theorem
\ref{thm:full}, and some remarks on the problem of computing $C\left(\Phi_{n}\right)$
and $C\left(IA\left(\Phi_{n}\right),\Phi_{n}\right)$.

\textbf{Acknowledgements:} I wish to offer my deepest thanks to my
great supervisor Prof. Alexander Lubotzky for his sensitive and devoted
guidance, and to the Rudin foundation trustees for their generous
support during the period of the research. 

\section{\label{sec:structure}Some properties of $IA\left(\Phi_{n}\right)$
and its subgroups}

Let $G=IA\left(\Phi_{n}\right)=\ker\left(Aut\left(\Phi_{n}\right)\to Aut\left(\Phi_{n}/\Phi'_{n}\right)=GL_{n}\left(\mathbb{Z}\right)\right)$.
We start with recalling some of the properties of $G=IA\left(\Phi_{n}\right)$
and its subgroups, as presented in Section 3 in \cite{key-14}. We
also refer the reader to \cite{key-14} for the proofs of the statements
in this section. We start with the following notations:
\begin{itemize}
\item $\Phi_{n}=F_{n}/F''_{n}$= the free metabelian group on $n$ elements.
Here $F''_{n}$ denotes the second derivative of $F_{n}$, the free
group on $n$ elements.
\item $\Phi_{n,m}=\Phi_{n}/M_{n,m}$, where $M_{n,m}=\left(\Phi'_{n}\Phi_{n}^{m}\right)'\left(\Phi'_{n}\Phi_{n}^{m}\right)^{m}$.
\item $IG_{n,m}=G(M_{n,m})=\ker\left(IA\left(\Phi_{n}\right)\to Aut\left(\Phi_{n,m}\right)\right).$
\item $IA_{n}^{m}=\left\langle IA\left(\Phi_{n}\right)^{m}\right\rangle $.
\item $R_{n}=\mathbb{Z}[\mathbb{Z}^{n}]=\mathbb{Z}[x_{1}^{\pm1},\ldots,x_{n}^{\pm1}]$
where $x_{1},\ldots,x_{n}$ are the generators of $\mathbb{Z}^{n}$.
\item $\mathbb{Z}_{m}=\mathbb{Z}/m\mathbb{Z}$.
\item $\sigma_{i}=x_{i}-1$ for $1\leq i\leq n$. We also denote by $\vec{\sigma}$
the column vector which has $\sigma_{i}$ in its $i$-th entry.
\item $\mathfrak{A}_{n}=\sum_{i=1}^{n}\sigma_{i}R_{n}$ = the augmentation
ideal of $R_{n}$.
\item $H_{n,m}=\ker\left(R_{n}\to\mathbb{Z}_{m}[\mathbb{Z}_{m}^{n}]\right)=\sum_{i=1}^{n}\left(x_{i}^{m}-1\right)R_{n}+mR_{n}$.
\end{itemize}
By the well known Magnus embedding (see \cite{key-36}, \cite{key-37},
\cite{key-35-1}), one can identify $\Phi_{n}$ with the matrix group
\[
\Phi_{n}=\left\{ \left(\begin{array}{cc}
g & a_{1}t_{1}+\ldots+a_{n}t_{n}\\
0 & 1
\end{array}\right)\,|\,g\in\mathbb{Z}^{n},\,a_{i}\in R_{n},\,g-1=\sum_{i=1}^{n}a_{i}(x_{i}-1)\right\} 
\]
where $t_{i}$ is a free basis for $R_{n}$-module, under the identification
of the generators of $\Phi_{n}$ with the matrices
\[
\left(\begin{array}{cc}
x_{i} & t_{i}\\
0 & 1
\end{array}\right)\,\,\,\,\,1\leq i\leq n.
\]
Moreover, for every $\alpha\in IA\left(\Phi_{n}\right)$, one can
describe $\alpha$ by its action on the generators of $\Phi_{n}$,
by
\[
\alpha:\left(\begin{array}{cc}
x_{i} & t_{i}\\
0 & 1
\end{array}\right)\mapsto\left(\begin{array}{cc}
x_{i} & a_{i,1}t_{1}+\ldots+a_{i,n}t_{n}\\
0 & 1
\end{array}\right)
\]
and this description gives an injective homomorphism (see \cite{key-13},
\cite{key-36})
\begin{eqnarray*}
IA\left(\Phi_{n}\right) & \hookrightarrow & GL_{n}\left(R_{n}\right)\\
\textrm{defined by}\,\,\,\,\alpha & \mapsto & \left(\begin{array}{ccc}
a_{1,1} & \cdots & a_{1,n}\\
\vdots &  & \vdots\\
a_{n,1} & \cdots & a_{n,n}
\end{array}\right)
\end{eqnarray*}
which gives an identification of $IA\left(\Phi_{n}\right)$ with the
subgroup
\begin{eqnarray*}
IA\left(\Phi_{n}\right) & = & \left\{ A\in GL_{n}\left(R_{n}\right)\,|\,A\vec{\sigma}=\vec{\sigma}\right\} \\
 & = & \left\{ I_{n}+A\in GL_{n}\left(R_{n}\right)\,|\,A\vec{\sigma}=\vec{0}\right\} .
\end{eqnarray*}

One can find the proof of the following proposition in \cite{key-14}
(Propositions 3.1 and 3.2):
\begin{prop}
\label{prop:augmentation}Let $I_{n}+A\in IA\left(\Phi_{n}\right)$.
Then:
\end{prop}

\begin{itemize}
\item \textit{If one denotes the entries of $A$ by $a_{k,l}$ for $1\leq k,l\leq n$,
then for every $1\leq k,l\leq n$, $a_{k,l}\in\sum_{l\neq i=1}^{n}\sigma_{i}R_{n}\subseteq\mathfrak{A}_{n}$.}
\item \textit{$\det\left(I_{n}+A\right)$ is of the form $\det\left(I_{n}+A\right)=\prod_{r=1}^{n}x_{r}^{s_{r}}$
for some $s_{r}\in\mathbb{Z}$.}
\end{itemize}
Consider now the map
\[
\begin{array}{c}
\Phi_{n}=\left\{ \left(\begin{array}{cc}
g & a_{1}t_{1}+\ldots+a_{n}t_{n}\\
0 & 1
\end{array}\right)\,|\,g\in\mathbb{Z}^{n},\,a_{i}\in R_{n},\,g-1=\sum_{i=1}^{n}a_{i}(x_{i}-1)\right\} \\
\downarrow\\
\left\{ \left(\begin{array}{cc}
g & a_{1}t_{1}+\ldots+a_{n}t_{n}\\
0 & 1
\end{array}\right)\,|\,g\in\mathbb{Z}_{m}^{n},\,a_{i}\in\mathbb{Z}_{m}[\mathbb{Z}_{m}^{n}],\,g-1=\sum_{i=1}^{n}a_{i}(x_{i}-1)\right\} 
\end{array}
\]
which induced by the projections $\mathbb{Z}^{n}\to\mathbb{Z}_{m}^{n}$,
$R_{n}=\mathbb{Z}[\mathbb{Z}^{n}]\to\mathbb{Z}_{m}[\mathbb{Z}_{m}^{n}]$.
Using result of Romanovski\u{\i} \cite{key-40}, it is shown in \cite{key-6}
that this map is surjective and that $\Phi_{n,m}$ is canonically
isomorphic to its image. Therefore, we can identify the principal
congruence subgroup of $IA\left(\Phi_{n}\right)$, $IG_{n,m}$, with
\begin{eqnarray*}
IG_{n,m} & = & \left\{ A\in\ker\left(GL_{n}\left(R_{n}\right)\to GL_{n}\left(\mathbb{Z}_{m}[\mathbb{Z}_{m}^{n}]\right)\right)\,|\,A\vec{\sigma}=\vec{\sigma}\right\} \\
 & = & \left\{ I_{n}+A\in GL_{n}\left(R_{n},H_{n,m}\right)\,|\,A\vec{\sigma}=\vec{0}\right\} .
\end{eqnarray*}

Let us step forward with the following definitions:
\begin{defn}
Let $A\in GL_{n}\left(R_{n}\right)$, and for $1\leq i\leq n$, denote
by $A_{i,i}$ the minor which obtained from $A$ by erasing its $i$-th
row and $i$-th column. Now, for every $1\leq i\leq n$, define the
subgroup $IGL_{n-1,i}\leq IA\left(\Phi_{n}\right)$, by
\[
IGL_{n-1,i}=\left\{ I_{n}+A\in IA\left(\Phi_{n}\right)\,|\,\begin{array}{c}
\textrm{The\,\,}i\textrm{-th\,\, row\,\, of\,\,}A\textrm{\,\, is\,\,0,}\\
I_{n-1}+A_{i,i}\in GL_{n-1}\left(R_{n},\sigma_{i}R_{n}\right)
\end{array}\right\} 
\]
where:
\[
GL_{n-1}\left(R_{n},\sigma_{i}R_{n}\right)=\ker(GL_{n-1}\left(R_{n}\right)\longrightarrow GL_{n-1}\left(R_{n}/\sigma_{i}R_{n}\right)).
\]

The following proposition is proven in \cite{key-14} (Proposition
3.4): 
\end{defn}

\begin{prop}
\label{prop:iso}For every $1\leq i\leq n$ we have $IGL_{n-1,i}\cong GL_{n-1}\left(R_{n},\sigma_{i}R_{n}\right)$.
\end{prop}

We recall the following definitions from Algebraic K-Theory:
\begin{defn}
Let $R$ be a commutative ring (with identity), $H\vartriangleleft R$
an ideal, and $d\in\mathbb{N}$. Then:
\end{defn}

\begin{itemize}
\item $E_{d}\left(R\right)=\left\langle I_{d}+rE_{i,j}\,|\,r\in R,\,1\leq i\neq j\leq d\right\rangle \leq SL_{d}\left(R\right)$
where $E_{i,j}$ is the matrix which has $1$ in the $\left(i,j\right)$-th
entry and $0$ elsewhere.
\item $SL_{d}\left(R,H\right)=\ker\left(SL_{d}\left(R\right)\to SL_{d}\left(R/H\right)\right)$.
\item $GL_{d}\left(R,H\right)=\ker\left(GL_{d}\left(R\right)\to GL_{d}\left(R/H\right)\right).$
\item $E_{d}\left(R,H\right)$ = the normal subgroup of $E_{d}\left(R\right)$,
which is generated as a normal subgroup by the elementary matrices
of the form $I_{d}+hE_{i,j}$ for $h\in H$.
\end{itemize}
Under the above identification of $IGL_{n-1,i}$ with $GL_{n-1}\left(R_{n},\sigma_{i}R_{n}\right)$,
for every $1\leq i\leq n$ we define:
\begin{defn}
Let $H\vartriangleleft R_{n}$. Then:
\begin{eqnarray*}
ISL_{n-1,i}\left(H\right) & = & IGL_{n-1,i}\cap SL_{n-1}\left(R_{n},H\right)\\
IE_{n-1,i}\left(H\right) & = & IGL_{n-1,i}\cap E{}_{n-1}\left(R_{n},H\right)\leq ISL_{n-1,i}\left(H\right).
\end{eqnarray*}
\end{defn}

\section{\label{sec:structure-1}The main theorem's proof}

Using the above notations we prove in $\varoint$\ref{sec:The-main-lemma}
the following main lemma:
\begin{lem}
\label{thm:stage 1}For every $n\geq4$ and $m\in\mathbb{N}$ one
has
\begin{eqnarray*}
IG_{n,m^{2}} & \subseteq & IA_{n}^{m}\cdot\prod_{i=1}^{n}ISL_{n-1,i}\left(\sigma_{i}H_{n,m}\right)\\
 & = & IA_{n}^{m}\cdot ISL_{n-1,1}\left(\sigma_{1}H_{n,m}\right)\cdot\ldots\cdot ISL_{n-1,n}\left(\sigma_{n}H_{n,m}\right).
\end{eqnarray*}
\end{lem}

Observe that it follows that when $n\geq4$, then for every $m\in\mathbb{N}$
\begin{eqnarray*}
IG_{n,m^{4}} & \subseteq & IA_{n}^{m^{2}}\cdot\prod_{i=1}^{n}ISL_{n-1,i}\left(\sigma_{i}H_{n,m^{2}}\right)\\
 & \subseteq & IA_{n}^{m}\cdot\prod_{i=1}^{n}ISL_{n-1,i}\left(\sigma_{i}H_{n,m^{2}}\right)\\
 & \subseteq & IA_{n}^{m}\cdot\prod_{i=1}^{n}ISL_{n-1,i}\left(H_{n,m^{2}}\right).
\end{eqnarray*}
The following Lemma is proved in \cite{key-14}, using classical results
from Algebraic K-theory (Lemma 7.1 in \cite{key-14}): 
\begin{lem}
\label{thm:stage 2}For every $n\geq4$, $1\leq i\leq n$ and $m\in\mathbb{N}$
one has
\[
IE_{n-1,i}\left(H_{n,m^{2}}\right)\subseteq IA_{n}^{m}.
\]
\end{lem}

Let us now quote the following proposition (see \cite{key-14}, Corollary
2.3):
\begin{prop}
\label{cor:important}Let $R$ be a commutative ring, $H\vartriangleleft R$
ideal of finite index and $d\geq3$. Assume also that $E_{d}\left(R\right)=SL_{d}\left(R\right)$.
Then:
\[
SK_{1}\left(R,H;d\right)=SL_{d}\left(R,H\right)/E{}_{d}\left(R,H\right)
\]
is a finite group which is central in $GL_{d}\left(R\right)/E{}_{d}\left(R,H\right)$.
\end{prop}

Now, according to Proposition \ref{cor:important} and the fact that
$E_{d}\left(R_{n}\right)=SL_{d}\left(R_{n}\right)$ for every $d\geq3$
\cite{key-33}, we obtain that for every $n\geq4$
\[
SL_{n-1}\left(R_{n},H_{n,m}\right)/E{}_{n-1}\left(R_{n},H_{n,m}\right)=SK_{1}\left(R,H_{n,m};n-1\right)
\]
is a finite group. Thus
\[
ISL_{n-1,i}\left(H_{n,m}\right)/IE_{n-1,i}\left(H_{n,m}\right)\leq SL_{n-1}\left(R_{n},H_{n,m}\right)/E{}_{n-1}\left(R_{n},H_{n,m}\right)
\]
is also a finite group. Hence, the conclusion from Lemmas \ref{thm:stage 1}
and \ref{thm:stage 2} is that for every $m\in\mathbb{N}$, one can
cover $IG_{n,m^{4}}$ with finite number of cosets of $IA_{n}^{m}$.
As $IG_{n,m^{4}}$ is obviously a finite index subgroup of $IA\left(\Phi_{n}\right)$
we deduce that $IA_{n}^{m}$ is also a finite index subgroup of $IA\left(\Phi_{n}\right)$.
Therefore, as every normal subgroup of $IA\left(\Phi_{n}\right)$
of index $m$ cotains $IA_{n}^{m}$ we deduce that one can write explicitely
$\widehat{IA\left(\Phi_{n}\right)}=\underleftarrow{\lim}\left(IA\left(\Phi_{n}\right)/IA_{n}^{m}\right)$.
On the other hand, it is proven in \cite{key-6} that $\hat{\Phi}_{n}=\underleftarrow{\lim}\Phi_{n,m}$,
and thus:
\begin{cor}
\label{cor:kernel}For every $n\geq4$ 
\begin{eqnarray*}
C\left(IA\left(\Phi_{n}\right),\Phi_{n}\right) & = & \ker\left(\underleftarrow{\lim}\left(IA\left(\Phi_{n}\right)/IA_{n}^{m}\right)\to\underleftarrow{\lim}\left(IA\left(\Phi_{n}\right)/IG_{n,m}\right)\right)\\
 & = & \ker\left(\underleftarrow{\lim}\left(IA\left(\Phi_{n}\right)/IA_{n}^{m}\right)\to\underleftarrow{\lim}\left(IA\left(\Phi_{n}\right)/IG_{n,m^{4}}\right)\right)\\
 & = & \underleftarrow{\lim}\left(IA_{n}^{m}\cdot IG_{n,m^{4}}/IA_{n}^{m}\right).
\end{eqnarray*}
\end{cor}

Now, Proposition \ref{cor:important} gives us also that for every
$m\in\mathbb{N}$ and $n\geq4$, the subgroup $SK_{1}\left(R_{n},H_{n,m};n-1\right)$
is central in $GL_{n-1}\left(R_{n}\right)/E{}_{n-1}\left(R_{n},H_{n,m}\right)$.
This fact is used in Section 5 of \cite{key-14} to prove that if
we define
\[
IA_{n,m}=\cap\left\{ N\vartriangleleft IA\left(\Phi_{n}\right)\,|\,[IA\left(\Phi_{n}\right):N]\,|\,m\right\} 
\]
then for every $n\geq4$, $m\in\mathbb{N}$ and $1\leq i\leq n$ the
subgroup
\[
IA_{n,m}\cdot ISL_{n-1,i}\left(\sigma_{i}H_{n,m^{2}}\right)/IA_{n,m}
\]
is central in $IA\left(\Phi_{n}\right)/IA_{n,m}$. Completely similar
arguments yield the following result\footnote{The only property of $IA_{n,m}$ used in Chapter 5 of \cite{key-14}
is that $IA_{n}^{m}\subseteq IA_{n,m}$.}:
\begin{prop}
\label{thm:stage 3}For every $n\geq4$, $m\in\mathbb{N}$ and $1\leq i\leq n$
the subgroup
\[
IA_{n}^{m}\cdot ISL_{n-1,i}\left(\sigma_{i}H_{n,m^{2}}\right)/IA_{n}^{m}
\]
is central in $IA\left(\Phi_{n}\right)/IA_{n}^{m}$.
\end{prop}

\begin{cor}
For every $n\geq4$ and $m\in\mathbb{N}$ the elements of the set
\[
IA_{n}^{m}\cdot\prod_{i=1}^{n}ISL_{n-1,i}\left(\sigma_{i}H_{n,m^{2}}\right)/IA_{n}^{m}
\]
belong to the center of $IA\left(\Phi_{n}\right)/IA_{n}^{m}$.
\end{cor}

The conclusion from the latter corollary is that for every $n\geq4$
and $m\in\mathbb{N}$, the set 
\[
IA_{n}^{m}\cdot\prod_{i=1}^{n}ISL_{n-1,i}\left(\sigma_{i}H_{n,m^{2}}\right)/IA_{n}^{m}
\]
is an $abelian$ $group$ which contained in the center of $IA\left(\Phi_{n}\right)/IA_{n}^{m}$.
In particular, $IA_{n}^{m}\cdot IG_{n,m^{4}}/IA_{n}^{m}$ is contained
in the center of $IA\left(\Phi_{n}\right)/IA_{n}^{m}$, and thus,
by Corollary \ref{cor:kernel}, $C\left(IA\left(\Phi_{n}\right),\Phi_{n}\right)$
is in the center of $\widehat{IA\left(\Phi_{n}\right)}$. This finishes,
up to the proof of Lemma \ref{thm:stage 1}, the proof of Theorem
\ref{thm:main}.

So it remains to prove Lemma \ref{thm:stage 1}. But before we start
to prove this lemma, we need to present some elements of $IA_{n}^{m}$.
We will do this in the following section.

\section{\label{sec:elementary}Some elementary elements of $\left\langle IA\left(\Phi_{n}\right)^{m}\right\rangle $}

In this section we introduce some elements of $IA_{n}^{m}=\left\langle IA\left(\Phi_{n}\right)^{m}\right\rangle $
which are needed through the proof of Lemma \ref{thm:stage 1}. As
one can see below, we separate the elementary elements to two types.
In addition, we separate the treatment of the elements of type 1,
to two parts. We hope this separation will make the process clearer. 

Additionally to the previous notations, on the section, and also later
on, we will use the notation
\[
\mu_{r,m}=\sum_{i=0}^{m-1}x_{r}^{i}\,\,\,\,\textrm{for}\,\,\,\,1\leq r\leq n.
\]

\subsection{Elementary elements of type 1}
\begin{prop}
\label{prop:type 1.1}Let $n\geq3$, $1\leq u\leq n$ and $m\in\mathbb{N}$.
Denote by $\vec{e}_{i}$ the $i$-th row standard vector. Then, the
elements of $IA\left(\Phi_{n}\right)$ of the form (the following
notation means that the matrix is similar to the identity matrix,
except the entries in the $u$-th row)
\[
\left(\begin{array}{ccccccc}
 & I_{u-1} &  & 0 &  & 0\\
a_{u,1} & \cdots & a_{u,u-1} & 1 & a_{u,u+1} & \cdots & a_{u,n}\\
 & 0 &  & 0 &  & I_{n-u}
\end{array}\right)\leftarrow u\textrm{-th\,\,\,\,\ row}
\]
when $\left(a_{u,1},\ldots,a_{u,u-1},0,a_{u,u+1},\ldots,a_{u,n}\right)$
is a linear combination of the vectors
\begin{eqnarray*}
 & 1. & \left\{ m\left(\sigma_{i}\vec{e}_{j}-\sigma_{j}\vec{e}_{i}\right)\,|\,i,j\neq u,\,i\neq j\right\} \\
 & 2. & \left\{ \sigma_{k}\mu_{k,m}\left(\sigma_{i}\vec{e}_{j}-\sigma_{j}\vec{e}_{i}\right)\,|\,i,j,k\neq u,\,i\neq j\right\} \\
 & 3. & \left\{ \sigma_{k}\mu_{i,m}\left(\sigma_{i}\vec{e}_{j}-\sigma_{j}\vec{e}_{i}\right)\,|\,i,j,k\neq u,\,i\neq j,\,k\neq j\right\} 
\end{eqnarray*}
with coefficients in $R_{n}$, belong to $IA_{n}^{m}$.
\end{prop}

Before proving this proposition, we present some more elements of
this type. Note that for the following proposition we assume $n\geq4$:
\begin{prop}
\label{prop:type 1.2}Let $n\geq4$, $1\leq u\leq n$ and $m\in\mathbb{N}$.
Then, the elements of $IA\left(\Phi_{n}\right)$ of the form
\[
\left(\begin{array}{ccccccc}
 & I_{u-1} &  & 0 &  & 0\\
a_{u,1} & \cdots & a_{u,u-1} & 1 & a_{u,u+1} & \cdots & a_{u,n}\\
 & 0 &  & 0 &  & I_{n-u}
\end{array}\right)\leftarrow u\textrm{-th\,\,\,\,\ row}
\]
when $\left(a_{u,1},\ldots,a_{u,u-1},0,a_{u,u+1},\ldots,a_{u,n}\right)$
is a linear combination of the vectors
\begin{eqnarray*}
 & 1. & \left\{ \sigma_{u}^{2}\mu_{u,m}\left(\sigma_{i}\vec{e}_{j}-\sigma_{j}\vec{e}_{i}\right)\,|\,i,j\neq u,\,i\neq j\right\} \\
 & 2. & \left\{ \sigma_{u}\sigma_{j}\mu_{i,m}\left(\sigma_{i}\vec{e}_{j}-\sigma_{j}\vec{e}_{i}\right)\,|\,i,j\neq u,\,i\neq j\right\} 
\end{eqnarray*}
with coefficients in $R_{n}$, belong to $IA_{n}^{m}$.
\end{prop}

\begin{proof}
(of Proposition \ref{prop:type 1.1}) Without loss of generality,
we assume that $u=1$. Observe now that for every $a_{i},b_{i}\in R_{n}$
for $2\leq i\leq n$ one has
\[
\left(\begin{array}{cccc}
1 & a_{2} & \cdots & a_{n}\\
0 &  & I_{n-1}
\end{array}\right)\left(\begin{array}{cccc}
1 & b_{2} & \cdots & b_{n}\\
0 &  & I_{n-1}
\end{array}\right)=\left(\begin{array}{cccc}
1 & a_{2}+b_{2} & \cdots & a_{n}+b_{n}\\
0 &  & I_{n-1}
\end{array}\right).
\]
Hence, it is enough to prove that the elements of the following forms
belong to $IA_{n}^{m}$ (when we write $a\vec{e}_{i}$ we mean that
the entry of the $i$-th column in the first row is $a$):
\begin{eqnarray*}
1. & \left(\begin{array}{cc}
1 & mf\left(\sigma_{i}\vec{e}_{j}-\sigma_{j}\vec{e}_{i}\right)\\
0 & I_{n-1}
\end{array}\right) & i,j\neq1,\,i\neq j,\,f\in R_{n}\\
2. & \left(\begin{array}{cc}
1 & \sigma_{k}\mu_{k,m}f\left(\sigma_{i}\vec{e}_{j}-\sigma_{j}\vec{e}_{i}\right)\\
0 & I_{n-1}
\end{array}\right) & i,j,k\neq1,\,i\neq j,\,f\in R_{n}\\
3. & \left(\begin{array}{cc}
1 & \sigma_{k}\mu_{i,m}f\left(\sigma_{i}\vec{e}_{j}-\sigma_{j}\vec{e}_{i}\right)\\
0 & I_{n-1}
\end{array}\right) & i,j,k\neq1,\,i\neq j,\,k\neq j,\,f\in R_{n}.
\end{eqnarray*}

We start with the elements of form 1. Here we have
\[
\left(\begin{array}{cc}
1 & mf\left(\sigma_{i}\vec{e}_{j}-\sigma_{j}\vec{e}_{i}\right)\\
0 & I_{n-1}
\end{array}\right)=\left(\begin{array}{cc}
1 & f\left(\sigma_{i}\vec{e}_{j}-\sigma_{j}\vec{e}_{i}\right)\\
0 & I_{n-1}
\end{array}\right)^{m}\in IA_{n}^{m}.
\]

We pass to the elements of form 2. In this case we have
\begin{eqnarray*}
IA_{n}^{m} & \ni & \left[\left(\begin{array}{cc}
1 & f\left(\sigma_{i}\vec{e}_{j}-\sigma_{j}\vec{e}_{i}\right)\\
0 & I_{n-1}
\end{array}\right)^{-1},\left(\begin{array}{cc}
x_{k} & -\sigma_{1}\vec{e}_{k}\\
0 & I_{n-1}
\end{array}\right)^{m}\right]\\
 & = & \left(\begin{array}{cc}
1 & \sigma_{k}\mu_{k,m}f\left(\sigma_{i}\vec{e}_{j}-\sigma_{j}\vec{e}_{i}\right)\\
0 & I_{n-1}
\end{array}\right).
\end{eqnarray*}

We finish with the elements of form 3. If $k=i$, it is a special
case of the previous case, so we assume $k\neq i$. So we assume that
$i,j,k$ are all different from each other and $i,j,k\neq1$ - observe
that this case is interesting only when $n\geq4$. The computation
here is more complicated than in the previous cases, so we will demonstrate
it for the special case: $n=4$, $i=2$, $j=3$, $k=4$. It is clear
that symmetrically, with similar argument, the same holds in general
when $n\geq4$ for every $i,j,k\neq1$ which different from each other.
So
\begin{eqnarray*}
IA_{4}^{m} & \ni & \left[\left(\begin{array}{cccc}
1 & 0 & -\sigma_{4}f & \sigma_{3}f\\
0 & 1 & 0 & 0\\
0 & 0 & 1 & 0\\
0 & 0 & 0 & 1
\end{array}\right),\left(\begin{array}{cccc}
1 & 0 & 0 & 0\\
0 & 1 & 0 & 0\\
0 & -\sigma_{3} & x_{2} & 0\\
0 & 0 & 0 & 1
\end{array}\right)^{-m}\right]\\
 & = & \left(\begin{array}{cccc}
1 & -\sigma_{4}f\mu_{2,m}\sigma_{3} & \sigma_{4}f\sigma_{2}\mu_{2,m} & 0\\
0 & 1 & 0 & 0\\
0 & 0 & 1 & 0\\
0 & 0 & 0 & 1
\end{array}\right).
\end{eqnarray*}
\end{proof}
We pass now to the proof of Proposition \ref{prop:type 1.2}.
\begin{proof}
(of Proposition \ref{prop:type 1.2}) Also here, without loss of generality,
we assume that $u=1$. Thus, all we need to show is that also the
elements of the following forms belong to $IA_{n}^{m}$:
\begin{eqnarray*}
1. & \left(\begin{array}{cc}
1 & \sigma_{1}^{2}\mu_{1,m}f\left(\sigma_{i}\vec{e}_{j}-\sigma_{j}\vec{e}_{i}\right)\\
0 & I_{n-1}
\end{array}\right) & i,j\neq1,\,i\neq j,\,f\in R_{n}\\
2. & \left(\begin{array}{cc}
1 & \sigma_{1}\sigma_{j}\mu_{i,m}f\left(\sigma_{i}\vec{e}_{j}-\sigma_{j}\vec{e}_{i}\right)\\
0 & I_{n-1}
\end{array}\right) & i,j\neq1,\,i\neq j,\,f\in R_{n}.
\end{eqnarray*}
Also here, to simplify the notations, we will demonstrate the proof
in the special case: $n=4$, $i=2$, $j=3$. We start with the first
form. From Proposition \ref{prop:type 1.1} we have (an element of
form 2 in Proposition \ref{prop:type 1.1})
\[
IA_{4}^{m}\ni\left(\begin{array}{cccc}
1 & 0 & 0 & 0\\
0 & 1 & 0 & 0\\
0 & 0 & 1 & 0\\
0 & \sigma_{3}\sigma_{1}\mu_{1,m}f & -\sigma_{2}\sigma_{1}\mu_{1,m}f & 1
\end{array}\right).
\]
Therefore, we also have
\begin{eqnarray*}
IA_{4}^{m} & \ni & \left[\left(\begin{array}{cccc}
x_{4} & 0 & 0 & -\sigma_{1}\\
0 & 1 & 0 & 0\\
0 & 0 & 1 & 0\\
0 & 0 & 0 & 1
\end{array}\right),\left(\begin{array}{cccc}
1 & 0 & 0 & 0\\
0 & 1 & 0 & 0\\
0 & 0 & 1 & 0\\
0 & \sigma_{3}\sigma_{1}\mu_{1,m}f & -\sigma_{2}\sigma_{1}\mu_{1,m}f & 1
\end{array}\right)\right]\\
 & = & \left(\begin{array}{cccc}
1 & -\sigma_{3}\sigma_{1}^{2}\mu_{1,m}f & \sigma_{2}\sigma_{1}^{2}\mu_{1,m}f & 0\\
0 & 1 & 0 & 0\\
0 & 0 & 1 & 0\\
0 & 0 & 0 & 1
\end{array}\right).
\end{eqnarray*}

We pass to the elements of form 2. From Proposition \ref{prop:type 1.1}
we have (an element of form 3 in Proposition \ref{prop:type 1.1})
\[
IA_{4}^{m}\ni\left(\begin{array}{cccc}
1 & 0 & 0 & 0\\
0 & 1 & 0 & 0\\
0 & 0 & 1 & 0\\
0 & \sigma_{1}\sigma_{3}\mu_{2,m}f & -\sigma_{1}\sigma_{2}\mu_{2,m}f & 1
\end{array}\right)
\]
and therefore, we have
\begin{eqnarray*}
IA_{4}^{m} & \ni & \left[\left(\begin{array}{cccc}
1 & 0 & \sigma_{4} & -\sigma_{3}\\
0 & 1 & 0 & 0\\
0 & 0 & 1 & 0\\
0 & 0 & 0 & 1
\end{array}\right),\left(\begin{array}{cccc}
1 & 0 & 0 & 0\\
0 & 1 & 0 & 0\\
0 & 0 & 1 & 0\\
0 & \sigma_{1}\sigma_{3}\mu_{2,m}f & -\sigma_{1}\sigma_{2}\mu_{2,m}f & 1
\end{array}\right)\right]\\
 & = & \left(\begin{array}{cccc}
1 & -\sigma_{1}\sigma_{3}^{2}\mu_{2,m}f & \sigma_{3}\sigma_{1}\sigma_{2}\mu_{2,m}f & 0\\
0 & 1 & 0 & 0\\
0 & 0 & 1 & 0\\
0 & 0 & 0 & 1
\end{array}\right).
\end{eqnarray*}
\end{proof}

\subsection{Elementary elements of type 2}
\begin{prop}
\label{prop:type 2}Let $n\geq4$, $1\leq u<v\leq n$ and $m\in\mathbb{N}$.
Then, the elements of $IA\left(\Phi_{n}\right)$ of the form
\[
\left(\begin{array}{ccccc}
I_{u-1} & 0 & 0 & 0 & 0\\
0 & 1+\sigma_{u}\sigma_{v}f & 0 & -\sigma_{u}^{2}f & 0\\
0 & 0 & I_{v-u-1} & 0 & 0\\
0 & \sigma_{v}^{2}f & 0 & 1-\sigma_{u}\sigma_{v}f & 0\\
0 & 0 & 0 & 0 & I_{n-v}
\end{array}\right)\begin{array}{c}
\leftarrow u\textrm{-th\,\,\,\,\ row}\\
\\
\leftarrow v\textrm{-th\,\,\,\,\ row}
\end{array}
\]
for $f\in H_{n,m}$, belong to $IA_{n}^{m}$.
\end{prop}

\begin{proof}
As before, to simplify the notations we will demonstrate the proof
in the case: $n=4$, $u=1$ and $v=2$, and it will be clear from
the computation that the same holds in the general case, provided
$n\geq4$. 

First observe that for every $f,g\in R_{n}$ we have
\[
\left(\begin{array}{cccc}
1+\sigma_{1}\sigma_{2}f & -\sigma_{1}^{2}f & 0 & 0\\
\sigma_{2}^{2}f & 1-\sigma_{1}\sigma_{2}f & 0 & 0\\
0 & 0 & 1 & 0\\
0 & 0 & 0 & 1
\end{array}\right)\left(\begin{array}{cccc}
1+\sigma_{1}\sigma_{2}g & -\sigma_{1}^{2}g & 0 & 0\\
\sigma_{2}^{2}g & 1-\sigma_{1}\sigma_{2}g & 0 & 0\\
0 & 0 & 1 & 0\\
0 & 0 & 0 & 1
\end{array}\right)
\]
\[
=\left(\begin{array}{cccc}
1+\sigma_{1}\sigma_{2}\left(f+g\right) & -\sigma_{1}^{2}\left(f+g\right) & 0 & 0\\
\sigma_{2}^{2}\left(f+g\right) & 1-\sigma_{1}\sigma_{2}\left(f+g\right) & 0 & 0\\
0 & 0 & 1 & 0\\
0 & 0 & 0 & 1
\end{array}\right)
\]
so it is enough to consider the cases $f\in mR_{4}$ and $f\in\sigma_{r}\mu_{r,m}R_{4}$
for $1\leq r\leq4$, separately. Consider now the following computation.
For an arbitrary $f\in R_{n}$ we have
\begin{eqnarray*}
 &  & \left[\left(\begin{array}{cccc}
1 & 0 & 0 & 0\\
0 & 1 & 0 & 0\\
0 & 0 & 1 & 0\\
-\sigma_{2}f & \sigma_{1}f & 0 & 1
\end{array}\right),\left(\begin{array}{cccc}
x_{4} & 0 & 0 & -\sigma_{1}\\
0 & x_{4} & 0 & -\sigma_{2}\\
0 & 0 & 1 & 0\\
0 & 0 & 0 & 1
\end{array}\right)^{-1}\right]\\
 &  & \cdot\left(\begin{array}{cccc}
1 & 0 & 0 & 0\\
0 & 1 & 0 & 0\\
0 & 0 & 1 & 0\\
-\sigma_{4}\sigma_{2}f & \sigma_{4}\sigma_{1}f & 0 & 1
\end{array}\right)=\left(\begin{array}{cccc}
1+\sigma_{1}\sigma_{2}f & -\sigma_{1}^{2}f & 0 & 0\\
\sigma_{2}^{2}f & 1-\sigma_{1}\sigma_{2}f & 0 & 0\\
0 & 0 & 1 & 0\\
0 & 0 & 0 & 1
\end{array}\right).
\end{eqnarray*}

Therefore, we conclude that if 
\[
\left(\begin{array}{cccc}
1 & 0 & 0 & 0\\
0 & 1 & 0 & 0\\
0 & 0 & 1 & 0\\
-\sigma_{4}\sigma_{2}f & \sigma_{4}\sigma_{1}f & 0 & 1
\end{array}\right),\,\left(\begin{array}{cccc}
1 & 0 & 0 & 0\\
0 & 1 & 0 & 0\\
0 & 0 & 1 & 0\\
-\sigma_{2}f & \sigma_{1}f & 0 & 1
\end{array}\right)\in IA_{4}^{m}
\]
then also
\[
\left(\begin{array}{cccc}
1+\sigma_{1}\sigma_{2}f & -\sigma_{1}^{2}f & 0 & 0\\
\sigma_{2}^{2}f & 1-\sigma_{1}\sigma_{2}f & 0 & 0\\
0 & 0 & 1 & 0\\
0 & 0 & 0 & 1
\end{array}\right)\in IA_{4}^{m}.
\]
Thus, the cases $f\in mR_{4}$ and $f\in\sigma_{r}\mu_{r,m}R_{4}$
for $r\neq4$, are obtained immediately from Proposition \ref{prop:type 1.1}.
Hence, it remains to deal with the case $f\in\sigma_{r}\mu_{r,m}R_{4}$
for $r=4$. However, it is easy to see that by switching the roles
of $3$ and $4$, the remained case is also obtained by similar arguments. 
\end{proof}

\section{\label{sec:The-main-lemma}A main lemma}

In this section we prove Lemma \ref{thm:stage 1} which states that
for every $n\geq4$ and $m\in\mathbb{N}$ we have
\begin{eqnarray*}
IG_{n,m^{2}} & \subseteq & IA_{n}^{m}\cdot\prod_{i=1}^{n}ISL_{n-1,i}\left(\sigma_{i}H_{n,m}\right)\\
 & = & IA_{n}^{m}\cdot ISL_{n-1,1}\left(\sigma_{1}H_{n,m}\right)\cdot\ldots\cdot ISL_{n-1,n}\left(\sigma_{n}H_{n,m}\right).
\end{eqnarray*}

The proof will be presented in a few stages - each of which will be
covered in a separate subsection. In this sections $n\geq4$ will
be constant, so we will make notations simpler and write
\[
\begin{array}{ccccc}
R=R_{n}, & \mathfrak{A}=\mathfrak{A}_{n}, & H_{m}=H_{n,m}, & IA^{m}=IA_{n}^{m}, & IG_{m}=IG_{n,m}.\end{array}
\]

We will also use the following notations:
\[
\begin{array}{cc}
O_{m}=mR, & U_{r,m}=\mu_{r,m}R\,\,\,\,\textrm{when\,\,}\,\,\mu_{r,m}=\sum_{i=0}^{m-1}x_{r}^{i}\,\,\,\,\textrm{for}\,\,\,\,1\leq r\leq n\end{array}.
\]

Notice that it follows from the definitions, that $H_{m}=\sum_{r=1}^{n}\sigma_{r}U_{r,m}+O_{m}$
(we note that in \cite{key-14} the notation $U_{r,m}$ is used for
$\sigma_{r}\mu_{r,m}R$).

Before we get deeply into the details, let us give an outline of the
proof of the above main lemma. Given $0\leq u\leq n$, denote the
ideal
\[
\mathfrak{\tilde{A}}_{u}=\sum_{r=u+1}^{n}\sigma_{r}R\vartriangleleft R.
\]
The lemma is proven by induction on $1\leq u\leq n$. Note that by
Proposition \ref{prop:augmentation} $IA(\Phi_{n})\subseteq GL_{n}(R,\mathfrak{A})=GL_{n}(R,\mathfrak{\tilde{A}}_{0})$.
Now, let $g\in IG_{n,m^{2}}\cap GL_{n}(R,\mathfrak{\tilde{A}}_{u-1})$.
If one could show that by multiplying it by elements of $IA^{m}$
and an element of $ISL_{n-1,u}\left(\sigma_{u}H_{n,m}\right)$ we
can ``push'' $g$ to  an element of $IG_{n,m^{2}}\cap GL_{n}(R,\mathfrak{\tilde{A}}_{u})$,
then as $GL_{n}(R,\mathfrak{\tilde{A}}_{n})=\{I_{n}\}$, it will certainly
be sufficient for proving the lemma. The issue is that the elements
of $IA^{m}$ take us out from $IG_{n,m^{2}}$. Hence, we extend $IG_{n,m^{2}}\cap GL_{n}(R,\mathfrak{\tilde{A}}_{u-1})$
to a larger subgroup, denoted by $\mathbb{\mathbb{\tilde{J}}}_{m,u-1}$.
In general, these subgroups do not satisfy $\mathbb{\mathbb{\tilde{J}}}_{m,u}\subseteq\mathbb{\mathbb{\tilde{J}}}_{m,u-1}$.
However, in the delicate process described below we show that we can
``push'' $g\in\mathbb{\mathbb{\tilde{J}}}_{m,u-1}$ to an element
of $\mathbb{\mathbb{\tilde{J}}}_{m,u}$ by elements of $IA^{m}$ and
an element of $ISL_{n-1,u}\left(\sigma_{u}H_{n,m}\right)$. We go
out from $\mathbb{\mathbb{\tilde{J}}}_{m,u-1}$ and get into $\mathbb{\mathbb{\tilde{J}}}_{m,u}$.
The process ends when we get into $\mathbb{\mathbb{\tilde{J}}}_{m,n}=\left\{ I_{n}\right\} $.
 The definition of $\mathbb{\mathbb{\tilde{J}}}_{m,u}$ is quite delicate,
and so is the process. 

In Subsection \ref{subsec:5.1} we describe the above definitions
and process is details. Then, in Subsection \ref{subsec:5.2} we show
that given an element of $\mathbb{\mathbb{\tilde{J}}}_{m,u-1}$, before
``pushing'' it into $\mathbb{\mathbb{\tilde{J}}}_{m,u}$, one can
fix it a bit with elements of $\mathbb{\mathbb{\tilde{J}}}_{m,u-1}\cap IA^{m}$
to a more convenient form. Then, in Subsection \ref{subsec:Finishing}
we define the ``pushing elements'' from $IA^{m}$ and $ISL_{n-1,u}\left(\sigma_{u}H_{n,m}\right)$.

\subsection{\label{subsec:5.1}Reducing Lemma \ref{thm:stage 1}'s proof}

We start this subsection with introducing the following objects:
\begin{defn}
Let $m\in\mathbb{{N}}$. Define
\begin{eqnarray*}
R\vartriangleright J_{m} & = & \sum_{r=1}^{n}\sigma_{r}^{3}U_{r,m}+\mathfrak{A}^{2}O_{m}+\mathfrak{A}O_{m}^{2}\\
\mathbb{\mathbb{J}}_{m} & = & \left\{ I_{n}+A\,|\,\begin{array}{c}
I_{n}+A\in IA\left(\Phi_{n}\right)\cap GL_{n}\left(R,J_{m}\right)\\
\det\left(I_{n}+A\right)=\prod_{r=1}^{n}x_{r}^{s_{r}m^{2}},\,\,s_{r}\in\mathbb{Z}
\end{array}\right\} .
\end{eqnarray*}
\end{defn}

\begin{prop}
\label{prop:reduction1}For every $m\in\mathbb{N}$ we have
\[
IG_{m^{2}}=IA\left(\Phi_{n}\right)\cap GL_{n}\left(R,H_{m^{2}}\right)\subseteq\mathbb{\mathbb{J}}_{m}.
\]
\end{prop}

\begin{proof}
Let $x\in R$. Notice that $\sum_{i=0}^{m-1}x^{i}\in\left(x-1\right)R+mR$.
In addition, by replacing $x$ by $x^{m}$ we obtain $\sum_{i=0}^{m-1}x^{mi}\in\left(x^{m}-1\right)R+mR$.
Hence
\begin{eqnarray*}
x^{m^{2}}-1 & = & \left(x-1\right)\sum_{i=0}^{m^{2}-1}x^{i}=\left(x-1\right)\sum_{i=0}^{m-1}x^{i}\sum_{i=0}^{m-1}x^{mi}\\
 & \in & \left(x-1\right)\left(\left(x-1\right)R+mR\right)\left(\left(x^{m}-1\right)R+mR\right)\\
 & \subseteq & \left(x-1\right)^{2}\left(x^{m}-1\right)R+\left(x-1\right)^{2}mR+\left(x-1\right)m^{2}R.
\end{eqnarray*}

Thus, we obtain that $H_{m^{2}}=\sum_{r=1}^{n}(x_{r}^{m^{2}}-1)R+m^{2}R\subseteq J_{m}+O_{m}^{2}$.
Now, let $I_{n}+A\in IG_{m^{2}}=IA\left(\Phi_{n}\right)\cap GL_{n}\left(R,H_{m^{2}}\right)$.
From the above observation and from Proposition \ref{prop:augmentation},
it follows that every entry of $A$ belongs to $\left(J_{m}+O_{m}^{2}\right)\cap\mathfrak{A}=J_{m}$.
In addition, by Proposition \ref{prop:augmentation}, the determinant
of $I_{n}+A$ is of the form $\prod_{r=1}^{n}x_{r}^{s_{r}}$. On the
other hand, we know that under the projection $R_{n}\to\mathbb{Z}_{m^{2}}[\mathbb{Z}_{m^{2}}^{n}]$
we have $I_{n}+A\mapsto I_{n}$ and thus also $\prod_{r=1}^{n}x_{r}^{s_{r}}=\det\left(I_{n}+A\right)\mapsto1$.
Therefore, $\det\left(I_{n}+A\right)$ is of the form $\prod_{r=1}^{n}x_{r}^{m^{2}s_{r}}$,
as required.
\end{proof}
\begin{cor}
\label{cor:first reduction}Let $n\geq4$ and $m\in\mathbb{N}$. Then,
for proving Lemma \ref{thm:stage 1} it suffices to prove that
\[
\mathbb{\mathbb{J}}_{m}\subseteq IA^{m}\cdot\prod_{i=1}^{n}ISL_{n-1,i}\left(\sigma_{i}H_{m}\right).
\]
\end{cor}

We continue with defining the following objects:
\begin{defn}
\label{def:objects}For $0\leq u\leq n$ and $1\leq v\leq n$, define
the following ideals of $R=R_{n}=\mathbb{Z}[x_{1}^{\pm1},\ldots,x_{n}^{\pm1}]$:
\begin{eqnarray*}
\mathfrak{\tilde{A}}_{u} & = & \sum_{r=u+1}^{n}\sigma_{r}R\\
\tilde{J}_{m,u,v} & = & \begin{cases}
\mathfrak{\tilde{A}}_{u}\left(\sum_{r=1}^{u}\mathfrak{A}\sigma_{r}U_{r,m}+\mathfrak{A}O_{m}+O_{m}^{2}\right)+\\
\sum_{r=u+1}^{n}\sigma_{r}^{3}U_{r,m} & v\leq u\\
\mathfrak{\tilde{A}}_{u}\left(\sum_{r=1}^{u}\mathfrak{A}\sigma_{r}U_{r,m}+\mathfrak{A}O_{m}+O_{m}^{2}\right)+\\
\sum_{v\neq r=u+1}^{n}\sigma_{r}^{3}U_{r,m}+\mathfrak{A}\sigma_{v}^{2}U_{v,m} & v>u
\end{cases}
\end{eqnarray*}
and for $0\leq u\leq n$ define the groups $\tilde{\mathbb{A}}_{u}=IA\left(\Phi_{n}\right)\cap GL_{n}(R,\mathfrak{\tilde{A}}_{u})$,
and
\[
\mathbb{\mathbb{\tilde{J}}}_{m,u}=\left\{ I_{n}+A\in IA\left(\Phi_{n}\right)\,|\,\begin{array}{c}
\det\left(I_{n}+A\right)=\prod_{i=1}^{n}x_{i}^{s_{i}m^{2}}\textrm{,\,\,every\,\,entry\,\,in}\\
\textrm{the\,\,}v\textrm{-th\,\,colmun\,\,of\,\,}A\textrm{\,\,belongs\,\,to\,\,}\tilde{J}_{m,u,v}
\end{array}\right\} .
\]
\end{defn}

\begin{rem}
If $I_{n}+A\in\mathbb{\mathbb{\tilde{J}}}_{m,u}$, the entries of
the columns of $A$ may belong to different ideals in $R$, so it
is not obvious that $\mathbb{\tilde{\mathbb{J}}}_{m,u}$ is indeed
a group, i.e. closed under matrix multiplication and the inverse operation.
However, showing that $\mathbb{\tilde{\mathbb{J}}}_{m,u}$ is a group
is not difficult and we leave it to the reader.
\end{rem}

Notice now the extreme cases: 

1. For $u=0$ we have (for every $v$ and $m$) $\mathfrak{\tilde{A}}_{0}=\mathfrak{A}$,
and $J_{m}\subseteq\tilde{J}_{m,0,v}$. Hence, we have $\mathbb{\mathbb{J}}_{m}\subseteq\mathbb{\mathbb{\tilde{J}}}_{m,0}$.

2. For $u=n$ we have (for every $v$ and $m$) $\mathfrak{\tilde{A}}_{n}=\tilde{J}_{m,n,v}=0$.
Hence, we also have $\mathbb{\mathbb{\tilde{J}}}_{m,n}=\left\{ I_{n}\right\} $.
\begin{cor}
\label{cor:reduction 2}For proving Lemma \ref{thm:stage 1}, it is
enough to prove that for every $1\leq u\leq n$
\[
\mathbb{\mathbb{\tilde{J}}}_{m,u-1}\subseteq IA^{m}\cdot ISL_{n-1,u}\left(\sigma_{u}H_{m}\right)\cdot\mathbb{\mathbb{\tilde{J}}}_{m,u}.
\]
\end{cor}

\begin{proof}
Using that $IA^{m}$ is normal in $IA\left(\Phi_{n}\right)$ and the
latter observations, under the above assumption, one obtains that
\begin{eqnarray*}
\mathbb{\mathbb{J}}_{m}\subseteq\mathbb{\mathbb{\tilde{J}}}_{m,0} & \subseteq & IA^{m}\cdot ISL_{n-1,1}\left(\sigma_{1}H_{m}\right)\cdot\mathbb{\mathbb{\tilde{J}}}_{m,1}\\
 & \subseteq & \ldots\\
 & \subseteq & \prod_{u=1}^{n}\left(IA^{m}\cdot ISL_{n-1,u}\left(\sigma_{u}H_{m}\right)\right)\cdot\mathbb{\mathbb{\tilde{J}}}_{m,n}\\
 & = & IA^{m}\prod_{u=1}^{n}ISL_{n-1,u}\left(\sigma_{u}H_{m}\right)
\end{eqnarray*}
which is the requirement of Corollary \ref{cor:first reduction}.
\end{proof}
We continue with defining the following objects:
\begin{defn}
For $0\leq u\leq n$ and $1\leq v\leq n$, define the following ideals
of $R$:
\begin{eqnarray*}
J_{m,u,v} & = & \begin{cases}
\mathfrak{A}\left(\sum_{r=1}^{u}\mathfrak{A}\sigma_{r}U_{r,m}+\mathfrak{A}O_{m}+O_{m}^{2}\right)+\\
\sum_{r=u+1}^{n}\sigma_{r}^{3}U_{r,m} & v\leq u\\
\mathfrak{A}\left(\sum_{r=1}^{u}\mathfrak{A}\sigma_{r}U_{r,m}+\mathfrak{A}O_{m}+O_{m}^{2}\right)+\\
\sum_{v\neq r=u+1}^{n}\sigma_{r}^{3}U_{r,m}+\mathfrak{A}\sigma_{v}^{2}U_{v,m} & v>u
\end{cases}
\end{eqnarray*}
and for $0\leq u\leq n$ define the group
\[
\mathbb{J}_{m,u}=\left\{ I_{n}+A\in IA\left(\Phi_{n}\right)\,|\,\begin{array}{c}
\det\left(I_{n}+A\right)=\prod_{i=1}^{n}x_{i}^{s_{i}m^{2}}\textrm{,\,\,every\,\,entry\,\,in}\\
\textrm{the\,\,}v\textrm{-th\,\,colmun\,\,of\,\,}A\textrm{\,\,belongs\,\,to\,\,}J_{m,u,v}
\end{array}\right\} .
\]
\end{defn}

It follows from the definitions that for every $1\leq u\leq n$ we
have:
\begin{enumerate}
\item $J_{m,u-1,v}\subseteq J_{m,u,v}$, but $\mathfrak{\tilde{A}}_{u-1}\supseteq\mathfrak{\tilde{A}}_{u}$.
Thus, we have also
\item $\mathbb{J}_{m,u-1}\subseteq\mathbb{J}_{m,u}$, but $\tilde{\mathbb{A}}_{u-1}\supseteq\tilde{\mathbb{A}}_{u}$.
\end{enumerate}
Here comes the connection between the latter objects to the objects
defined in Definition \ref{def:objects}.
\begin{prop}
\label{lem:connection}For every $0\leq u\leq n$ and $1\leq v\leq n$
we have $J_{m,u,v}\cap\mathfrak{\tilde{A}}_{u}=\tilde{J}_{m,u,v}$,
and hence $\mathbb{J}_{m,u}\cap\tilde{\mathbb{A}}_{u}=\mathbb{\mathbb{\tilde{\mathbb{J}}}}_{m,u}$.
\end{prop}

\begin{proof}
It is clear from the definitions that we have $\tilde{J}_{m,u,v}\subseteq J_{m,u,v}\cap\mathfrak{\tilde{A}}_{u}$,
so we have to show an opposite inclusion. Let $a\in J_{m,u,v}\cap\mathfrak{\tilde{A}}_{u}$.
As 
\[
\tilde{J}_{m,u,v}\supseteq\begin{cases}
\sum_{r=u+1}^{n}\sigma_{r}^{3}U_{r,m} & v\leq u\\
\sum_{v\neq r=u+1}^{n}\sigma_{r}^{3}U_{r,m}+\mathfrak{A}\sigma_{v}^{2}U_{v,m} & v>u
\end{cases}
\]
we can assume that $a\in\mathfrak{A}\left(\sum_{r=1}^{u}\mathfrak{A}\sigma_{r}U_{r,m}+\mathfrak{A}O_{m}+O_{m}^{2}\right)\cap\mathfrak{\tilde{A}}_{u}$. 

Observe now that by dividing an element $b\in R$ by $\sigma_{u+1},\ldots,\sigma_{n}$
(with residue), one can present $b$ as a summand of an element of
$\mathfrak{\tilde{A}}_{u}$ with an element of $R_{u}=\mathbb{Z}[x_{1}^{\pm1},\ldots,x_{u}^{\pm1}]$.
Hence, $R=\mathfrak{\tilde{A}}_{u}+R_{u}$ and $\mathfrak{A}=\mathfrak{\tilde{A}}_{u}+\mathfrak{A}_{u}$,
where $\mathfrak{A}_{u}$ is the augmentation ideal of $R_{u}$. Hence
\begin{eqnarray*}
a & \in & (\mathfrak{\tilde{A}}_{u}+\mathfrak{A}_{u})^{2}\sum_{r=1}^{u}\sigma_{r}\mu_{r,m}(\mathfrak{\tilde{A}}_{u}+R_{u})\\
 &  & +\,(\mathfrak{\tilde{A}}_{u}+\mathfrak{A}_{u})^{2}m(\mathfrak{\tilde{A}}_{u}+R_{u})+(\mathfrak{\tilde{A}}_{u}+\mathfrak{A}_{u})m^{2}(\mathfrak{\tilde{A}}_{u}+R_{u})\\
 & \subseteq & \tilde{J}_{m,u,v}+\mathfrak{A}_{u}^{2}\sum_{r=1}^{u}\sigma_{r}\mu_{r,m}R_{u}+\mathfrak{A}_{u}^{2}mR_{u}+\mathfrak{A}_{u}m^{2}R_{u}.
\end{eqnarray*}
Hence, we can assume that $a\in\left(\mathfrak{A}_{u}^{2}\sum_{r=1}^{u}\sigma_{r}\mu_{r,m}R_{u}+\mathfrak{A}_{u}^{2}mR_{u}+\mathfrak{A}_{u}m^{2}R_{u}\right)\cap\mathfrak{\tilde{A}}_{u}=\left\{ 0\right\} $,
i.e. $a=0\in\tilde{J}_{m,u,v}$, as required.
\end{proof}
Due to the above, we can now reduce Lemma \ref{thm:stage 1}'s proof
as follows.
\begin{cor}
\label{cor:reduction}For proving Lemma \ref{thm:stage 1} it suffices
to show that given $1\leq u\leq n$, for every $\alpha\in\mathbb{\tilde{\mathbb{J}}}_{m,u-1}$
there exist $\beta\in IA^{m}\cap\mathbb{J}_{m,u}$ and $\gamma\in ISL_{n-1,u}\left(\sigma_{u}H_{m}\right)\cap\mathbb{\mathbb{J}}_{m,u}$
such that $\gamma\alpha\beta\in\tilde{\mathbb{A}}{}_{u}$.
\end{cor}

\begin{proof}
As clearly $\mathbb{J}_{m,u}\supseteq\mathbb{\mathbb{J}}_{m,u-1}\supseteq\mathbb{\mathbb{\tilde{J}}}_{m,u-1}$,
we obtain from Proposition \ref{lem:connection} that $\gamma\alpha\beta\in\tilde{\mathbb{A}}{}_{u}\cap\mathbb{J}_{m,u}=\mathbb{\mathbb{\tilde{J}}}_{m,u}$.
Thus
\[
\alpha\in ISL_{n-1,u}\left(\sigma_{u}H_{m}\right)\cdot\mathbb{\mathbb{\tilde{J}}}_{m,u}\cdot IA^{m}=IA^{m}\cdot ISL_{n-1,u}\left(\sigma_{u}H_{m}\right)\cdot\mathbb{\mathbb{\tilde{J}}}_{m,u}.
\]
This yields that $\mathbb{\mathbb{\tilde{J}}}_{m,u-1}\subseteq IA^{m}\cdot ISL_{n-1,u}\left(\sigma_{u}H_{m}\right)\cdot\tilde{\mathbb{\mathbb{J}}}_{m,u}$
which is the requirement of Corollary \ref{cor:reduction 2}.
\end{proof}

\subsection{\label{subsec:5.2}A technical lemma}

In this section we will prove a technical lemma, which will help us
in subsection \ref{subsec:Finishing} to prove Lemma \ref{thm:stage 1}.
In the following subsections $1\leq u\leq n$ will be constant. We
will use the following notations:
\begin{itemize}
\item For $a\in R$ we denote its image in $R_{u}$ under the projection
$x_{u+1},\ldots,x_{n}\mapsto1$ by $\bar{a}$. In addition, we denote
its image in $R_{u-1}$ under the projection $x_{u},\ldots,x_{n}\mapsto1$
by $\bar{\bar{a}}$.
\item For $\alpha\in GL_{n}\left(R\right)$ we denote its image in $GL_{n}\left(R_{u}\right)$
under the projection $x_{u+1},\ldots,x_{n}\mapsto1$ by $\bar{\alpha}$.
\item Similarly, we will use the following notations for every $m\in\mathbb{N}$:

\begin{itemize}
\item $\mathfrak{\bar{A}}=\mathfrak{A}_{u}=\sum_{i=1}^{u}\sigma_{i}R_{u}$,
$\bar{U}_{r,m}=\mu_{r,m}R_{u}$ for $1\leq r\leq u$, $\bar{O}_{m}=mR_{u}$
and $\bar{H}_{m}=H_{u,m}=\sum_{r=1}^{u}\sigma_{r}\bar{U}_{r,m}+\bar{O}_{m}$.
\item $\bar{\bar{\mathfrak{A}}}=\mathfrak{A}_{u-1}=\sum_{i=1}^{u-1}\sigma_{i}R_{u-1}$,
$\bar{\bar{U}}_{r,m}=\mu_{r,m}R_{u-1}$ for $1\leq r\leq u-1$ and
$\bar{\bar{O}}_{m}=mR_{u-1}$.
\end{itemize}
\end{itemize}
Now, let $\alpha=I_{n}+A\in\mathbb{\mathbb{\tilde{J}}}_{m,u-1}$,
and denote the entries of $A$ by $a_{i,j}$. Consider the $u$-th
row of $A$. Under the above assumption, for every $v$ we have
\[
a_{u,v}\in\begin{cases}
\mathfrak{\tilde{A}}_{u-1}\left(\sum_{r=1}^{u-1}\mathfrak{A}\sigma_{r}U_{r,m}+\mathfrak{A}O_{m}+O_{m}^{2}\right)+\\
\sum_{r=u}^{n}\sigma_{r}^{3}U_{r,m} & v<u\\
\mathfrak{\tilde{A}}_{u-1}\left(\sum_{r=1}^{u-1}\mathfrak{A}\sigma_{r}U_{r,m}+\mathfrak{A}O_{m}+O_{m}^{2}\right)+\\
\sum_{v\neq r=u}^{n}\sigma_{r}^{3}U_{r,m}+\mathfrak{A}\sigma_{v}^{2}U_{v,m} & v\geq u.
\end{cases}
\]
Hence we have
\begin{equation}
\bar{a}_{u,v}\in\begin{cases}
\sigma_{u}\left(\sum_{r=1}^{u-1}\mathfrak{\bar{A}}\sigma_{r}\bar{U}_{r,m}+\bar{\mathfrak{A}}\bar{O}_{m}+\bar{O}_{m}^{2}\right)+\mathfrak{\bar{A}}\sigma_{u}^{2}\bar{U}_{u,m}\\
=\sigma_{u}\left(\sum_{r=1}^{u}\mathfrak{\bar{A}}\sigma_{r}\bar{U}_{r,m}+\mathfrak{\bar{A}}\bar{O}_{m}+\bar{O}_{m}^{2}\right) & v=u\\
\sigma_{u}\left(\sum_{r=1}^{u-1}\mathfrak{\bar{A}}\sigma_{r}\bar{U}_{r,m}+\mathfrak{\bar{A}}\bar{O}_{m}+\bar{O}_{m}^{2}\right)+\sigma_{u}^{3}\bar{U}_{u,m}\\
=\sigma_{u}\left(\sum_{r=1}^{u-1}\mathfrak{\bar{A}}\sigma_{r}\bar{U}_{r,m}+\sigma_{u}^{2}\bar{U}_{u,m}+\mathfrak{\bar{A}}\bar{O}_{m}+\bar{O}_{m}^{2}\right) & v\neq u.
\end{cases}\label{eq:reminder}
\end{equation}

We can state now the technical lemma:
\begin{lem}
Let $\alpha=I_{n}+A\in\mathbb{\tilde{\mathbb{J}}}_{m,u-1}$. Then,
there exists $\delta\in IA^{m}\cap\mathbb{\mathbb{\tilde{J}}}_{m,u-1}$
such that for every $v\neq u$, the $\left(u,v\right)$-th entry of
$\overline{\alpha\delta^{-1}}$ belongs to $\sigma_{u}^{2}\bar{H}_{m}$.
\end{lem}

We will prove the lemma in two steps. Here is the first step:
\begin{prop}
\label{prop:step1}Let $\alpha=I_{n}+A\in\mathbb{\tilde{\mathbb{J}}}_{m,u-1}$.
Then, there exists $\delta\in IA^{m}\cap\mathbb{\mathbb{\tilde{J}}}_{m,u-1}$
such that for every $v<u$, the $\left(u,v\right)$-th entry of $\overline{\alpha\delta^{-1}}$
belongs to $\sigma_{u}^{2}\bar{H}_{m}$.
\end{prop}

\begin{proof}
So let $\alpha=I_{n}+A\in\mathbb{\mathbb{\tilde{J}}}_{m,u-1}$, and
observe that for every $1\leq v\leq u-1$ one can write $\bar{a}_{u,v}=\sigma_{u}\bar{b}_{u,v}$
for some $\bar{b}_{u,v}\in\sum_{r=1}^{u-1}\mathfrak{\bar{A}}\sigma_{r}\bar{U}_{r,m}+\sigma_{u}^{2}\bar{U}_{u,m}+\mathfrak{\bar{A}}\bar{O}_{m}+\bar{O}_{m}^{2}$.
In addition, as it is easy to see that
\[
\sum_{r=1}^{u-1}\mathfrak{\bar{A}}\sigma_{r}\bar{U}_{r,m}=\sum_{r=1}^{u-1}(\sigma_{u}R_{u}+\bar{\bar{\mathfrak{A}}})\sigma_{r}(\sigma_{u}\bar{U}_{r,m}+\bar{\bar{U}}_{r,m})\subseteq\sigma_{u}\sum_{r=1}^{u-1}\sigma_{r}\bar{U}_{r,m}+\sum_{r=1}^{u-1}\bar{\bar{\mathfrak{A}}}\sigma_{r}\bar{\bar{U}}_{r,m}
\]
\[
\mathfrak{\bar{A}}\bar{O}_{m}+\bar{O}_{m}^{2}=(\sigma_{u}R_{u}+\bar{\bar{\mathfrak{A}}})(\sigma_{u}\bar{O}_{m}+\bar{\bar{O}}_{m})+(\sigma_{u}\bar{O}_{m}+\bar{\bar{O}}_{m})^{2}\subseteq\sigma_{u}\bar{O}_{m}+\bar{\bar{\mathfrak{A}}}\bar{\bar{O}}_{m}+\bar{\bar{O}}_{m}^{2}
\]
one can write $\bar{b}_{u,v}=\sigma_{u}\bar{c}_{u,v}+\bar{\bar{b}}_{u,v}$
for every $1\leq v\leq u-1$, for some 
\begin{eqnarray*}
\bar{\bar{b}}_{u,v} & \in & \sum_{r=1}^{u-1}\bar{\bar{\mathfrak{A}}}\sigma_{r}\bar{\bar{U}}_{r,m}+\bar{\bar{\mathfrak{A}}}\bar{\bar{O}}_{m}+\bar{\bar{O}}_{m}^{2}\\
\bar{c}_{u,v} & \in & \sum_{r=1}^{u-1}\sigma_{r}\bar{U}_{r,m}+\sigma_{u}\bar{U}_{u,m}+\bar{O}_{m}=\bar{H}_{m}.
\end{eqnarray*}

Notice, that as $A$ satisfies the condition $A\vec{\sigma}=\vec{0}$
we have the equality $\sigma_{1}a_{u,1}+\ldots+\sigma_{n}a_{u,n}=0$,
which yields the following equalities as well:
\begin{eqnarray*}
\sigma_{1}\bar{a}_{u,1}+\ldots+\sigma_{u-1}\bar{a}_{u,u-1}+\sigma_{u}\bar{a}_{u,u} & = & 0\\
 & \Downarrow\\
\sigma_{1}\bar{b}_{u,1}+\ldots+\sigma_{u-1}\bar{b}_{u,u-1}+\bar{a}_{u,u} & = & 0\\
 & \Downarrow\\
\sigma_{1}\bar{\bar{b}}_{u,1}+\ldots+\sigma_{u-1}\bar{\bar{b}}_{u,u-1} & = & 0.
\end{eqnarray*}

Observe now that for every $1\leq v\leq u-1$ we have
\[
\sigma_{u}\bar{\bar{b}}_{u,v}\in\sigma_{u}\left(\sum_{r=1}^{u-1}\bar{\bar{\mathfrak{A}}}\sigma_{r}\bar{\bar{U}}_{r,m}+\bar{\bar{\mathfrak{A}}}\bar{\bar{O}}_{m}+\bar{\bar{O}}_{m}^{2}\right)\subseteq\tilde{J}_{m,u-1,v}
\]
and thus, if we define
\[
\delta=\left(\begin{array}{ccccc}
 & I_{u-1} &  & 0 & 0\\
\sigma_{u}\bar{\bar{b}}_{u,1} & \cdots & \sigma_{u}\bar{\bar{b}}_{u,u-1} & 1 & 0\\
 & 0 &  & 0 & I_{n-u}
\end{array}\right)\leftarrow u\textrm{-th}\,\,\,\,\textrm{row}
\]
then $\delta\in\tilde{\mathbb{\mathbb{J}}}_{m,u-1}$. We claim now
that we also have $\delta\in IA^{m}$. We will prove this claim soon,
but assuming this claim, we can now multiply $\alpha$ from the right
by $\delta^{-1}\in\tilde{\mathbb{\mathbb{J}}}_{m,u-1}\cap IA^{m}$
and obtain an element in $\mathbb{\mathbb{\tilde{J}}}_{m,u-1}$ such
that the image of its $\left(u,v\right)$-th entry for $1\leq v\leq u-1$,
under the projection $x_{u+1},\ldots,x_{n}\mapsto1$, is
\begin{eqnarray*}
\bar{a}_{u,v}-\sigma_{u}\bar{\bar{b}}_{u,v}\left(1+\bar{a}_{u,u}\right) & = & \sigma_{u}^{2}\bar{c}_{u,v}-\sigma_{u}\bar{\bar{b}}_{u,v}\bar{a}_{u,u}\\
 & \in & \sigma_{u}^{2}\bar{H}_{m}+\sigma_{u}^{2}\left(\sum_{r=1}^{u-1}\bar{\bar{\mathfrak{A}}}\sigma_{r}\bar{\bar{U}}_{r,m}+\bar{\bar{\mathfrak{A}}}\bar{\bar{O}}_{m}+\bar{\bar{O}}_{m}^{2}\right)\\
 & = & \sigma_{u}^{2}\bar{H}_{m}
\end{eqnarray*}
as required.
\end{proof}
So it remains to prove the following claim:
\begin{claim}
\label{claim:}Let $n\geq4$, $1\leq u\leq n$, and $\bar{\bar{b}}_{u,v}\in\sum_{r=1}^{u-1}\bar{\bar{\mathfrak{A}}}\sigma_{r}\bar{\bar{U}}_{r,m}+\bar{\bar{\mathfrak{A}}}\bar{\bar{O}}_{m}+\bar{\bar{O}}_{m}^{2}$
for $1\leq v\leq u-1$ which satisfy the condition
\begin{equation}
\sigma_{1}\bar{\bar{b}}_{u,1}+\ldots+\sigma_{u-1}\bar{\bar{b}}_{u,u-1}=0.\label{eq:condition--}
\end{equation}
Then
\[
u\textrm{-th\,\,\,\,row}\rightarrow\left(\begin{array}{ccccc}
 & I_{u-1} &  & 0 & 0\\
\sigma_{u}\bar{\bar{b}}_{u,1} & \cdots & \sigma_{u}\bar{\bar{b}}_{u,u-1} & 1 & 0\\
 & 0 &  & 0 & I_{n-u}
\end{array}\right)\in IA^{m}.
\]
\end{claim}

\begin{proof}
It will be easier to prove a bit more - we will prove that if for
every $1\leq v\leq u-1$
\[
\bar{\bar{b}}_{u,v}\in\sum_{v\neq r=1}^{u-1}\bar{\bar{\mathfrak{A}}}\sigma_{r}\bar{\bar{U}}_{r,m}+\bar{\bar{\mathfrak{A}}}^{2}\bar{\bar{U}}_{v,m}+\bar{\bar{O}}_{m}
\]
then the vector $\vec{b}=(\bar{\bar{b}}_{u,1},\ldots,\bar{\bar{b}}_{u,u-1},0,\ldots,0)$
is a linear combination of the vectors
\[
\left\{ \begin{array}{c}
\sigma_{k}\mu_{k,m}\left(\sigma_{i}\vec{e}_{j}-\sigma_{j}\vec{e}_{i}\right)\\
\sigma_{k}\mu_{i,m}\left(\sigma_{i}\vec{e}_{j}-\sigma_{j}\vec{e}_{i}\right)
\end{array},m\left(\sigma_{i}\vec{e}_{j}-\sigma_{j}\vec{e}_{i}\right)\,|\,i,j,k\leq u-1,\,i\neq j\right\} 
\]
with coefficients in $R_{u-1}.$ This will show that $\sigma_{u}(\bar{\bar{b}}_{u,1},\ldots,\bar{\bar{b}}_{u,u-1},0,\ldots,0)$
is a linear combination of the vectors in Propositions \ref{prop:type 1.1}
and \ref{prop:type 1.2}, so the claim will follow.

We start with expressing $\bar{\bar{b}}_{u,1}$ explicitly by writing
\[
\bar{\bar{b}}_{u,1}=\sum_{r=2}^{u-1}\sum_{i=1}^{u-1}\sigma_{i}\sigma_{r}\mu_{r,m}p_{i,r}+\sum_{i,j=1}^{u-1}\sigma_{i}\sigma_{j}\mu_{1,m}q_{i,j}+mr
\]
for some $p_{i,r},\,q_{i,j},\,r\in R_{u-1}$. Now, Equation \ref{eq:condition--}
gives that under the projection $\sigma_{2},\ldots,\sigma_{u-1}\mapsto0$,
$\bar{\bar{b}}_{u,1}\mapsto0$. It follows that $\bar{\bar{b}}_{u,1}\in\sum_{i=2}^{u-1}\sigma_{i}R_{u-1}\subseteq\bar{\bar{\mathfrak{A}}}$.
In particular, as obviously
\[
\sum_{r=2}^{u-1}\sum_{i=1}^{u-1}\sigma_{i}\sigma_{r}\mu_{r,m}p_{i,r}+\sum_{i,j=1}^{u-1}\sigma_{i}\sigma_{j}\mu_{1,m}q_{i,j}\in\bar{\bar{\mathfrak{A}}}
\]
we also have $mr\in\bar{\bar{\mathfrak{A}}}$ and hence $r\in\bar{\bar{\mathfrak{A}}}$.
Hence, we can write
\[
\bar{\bar{b}}_{u,1}=\sum_{r=2}^{u-1}\sum_{i=1}^{u-1}\sigma_{i}\sigma_{r}\mu_{r,m}p_{i,r}+\sum_{i,j=1}^{u-1}\sigma_{i}\sigma_{j}\mu_{1,m}q_{i,j}+\sum_{i=1}^{u-1}\sigma_{i}mr_{i}
\]
for some $p_{i,r},\,q_{i,j},\,r_{i}\in R_{u-1}$.

Observe now that by dividing $r_{1}$ by $\sigma_{2},\ldots,\sigma_{u-1}$
(with residue) we can write $r_{1}=r'_{1}+\sum_{i=2}^{u-1}\sigma_{i}r'_{i}$
where $r'_{1}$ depends only on $x_{1}$. Therefore, by replacing
$r_{1}$ by $r'_{1}$ and $r_{i}$ by $r_{i}+\sigma_{1}r'_{i}$ for
$2\leq i\leq n$, we can assume that $r_{1}$ depends only on $x_{1}$.
Similarly, by dividing $q_{1,1}$ by $\sigma_{2},\ldots,\sigma_{u-1}$,
we can assume that $q_{1,1}$ depends only on $x_{1}$. Now, by replacing
$\vec{b}$ with
\begin{eqnarray*}
\vec{b} & - & \sum_{r=2}^{u-1}\sum_{i=1}^{u-1}\sigma_{i}\mu_{r,m}p_{i,r}\left(\sigma_{r}\vec{e}_{1}-\sigma_{1}\vec{e}_{r}\right)\\
 & - & \sum_{i=2}^{u-1}\sum_{j=1}^{u-1}\sigma_{j}\mu_{1,m}q_{i,j}\left(\sigma_{i}\vec{e}_{1}-\sigma_{1}\vec{e}_{i}\right)-\sum_{j=2}^{u-1}\sigma_{1}\mu_{1,m}q_{1,j}\left(\sigma_{j}\vec{e}_{1}-\sigma_{1}\vec{e}_{j}\right)\\
 & - & \sum_{i=2}^{u-1}mr_{i}\left(\sigma_{i}\vec{e}_{1}-\sigma_{1}\vec{e}_{i}\right)
\end{eqnarray*}
we can assume that $\bar{\bar{b}}_{u,1}$ is a polynomial which depends
only on $x_{1}$. On the other hand, we already saw that Equation
\ref{eq:condition--} yields that $\bar{\bar{b}}_{u,1}\in\sum_{i=2}^{u-1}\sigma_{i}R_{u-1}$,
so we can actually assume that $\bar{\bar{b}}_{u,1}=0$.

We continue in this manner by induction. In the $1\leq v\leq u-1$
stage we assume that $\bar{\bar{b}}_{u,1}=\ldots\bar{\bar{b}}_{u,v-1}=0$.
Then we write
\[
\bar{\bar{b}}_{u,v}=\sum_{v\neq r=1}^{u-1}\sum_{i=1}^{u-1}\sigma_{i}\sigma_{r}\mu_{r,m}p_{i,r}+\sum_{i,j=1}^{u-1}\sigma_{i}\sigma_{j}\mu_{v,m}q_{i,j}+mr
\]
for some $p_{i,r},\,q_{i,j},\,r\in R_{u-1}$. The condition $\bar{\bar{b}}_{u,1}=\ldots=\bar{\bar{b}}_{u,v-1}=0$
and Equation \ref{eq:condition--} give that $\sigma_{v}\bar{\bar{b}}_{u,v}+\sigma_{v+1}\bar{\bar{b}}_{u,v+1}+\ldots+\sigma_{u-1}\bar{\bar{b}}_{u,u-1}=0$
and thus, under the projection $\sigma_{v+1},\ldots,\sigma_{u-1}\mapsto0$,
$\bar{\bar{b}}_{u,v}\mapsto0$, so $\bar{\bar{b}}_{u,v}\in\sum_{i=v+1}^{u-1}\sigma_{i}R_{u-1}\subseteq\bar{\bar{\mathfrak{A}}}$.
In particular, $r\in\bar{\bar{\mathfrak{A}}}$, so we can write
\[
\bar{\bar{b}}_{u,v}=\sum_{v\neq r=1}^{u-1}\sum_{i=1}^{u-1}\sigma_{i}\sigma_{r}\mu_{r,m}p_{i,r}+\sum_{i,j=1}^{u-1}\sigma_{i}\sigma_{j}\mu_{v,m}q_{i,j}+\sum_{i=1}^{u-1}\sigma_{i}mr_{i}
\]
for some $p_{i,r},\,q_{i,j},\,r_{i}\in R_{u-1}$.

Now, as we explained previously, by dividing $p_{i,r},\,q_{i,j},\,r_{i}$
for $1\leq i,j,r\leq v$ by $\sigma_{v+1},\ldots,\sigma_{u-1}$, we
can assume that these polynomials depend only on $x_{1},\ldots,x_{v}$.
Thus, by replacing $\vec{b}$ with
\begin{eqnarray*}
\vec{b} & - & \sum_{r=v+1}^{u-1}\sum_{i=1}^{u-1}\sigma_{i}\mu_{r,m}p_{i,r}\left(\sigma_{r}\vec{e}_{v}-\sigma_{v}\vec{e}_{r}\right)-\sum_{r=1}^{v-1}\sum_{i=v+1}^{u-1}\sigma_{r}\mu_{r,m}p_{i,r}\left(\sigma_{i}\vec{e}_{v}-\sigma_{v}\vec{e}_{i}\right)\\
 & - & \sum_{i=v+1}^{u-1}\sum_{j=1}^{u-1}\sigma_{j}\mu_{v,m}q_{i,j}\left(\sigma_{i}\vec{e}_{v}-\sigma_{v}\vec{e}_{i}\right)-\sum_{i=1}^{v}\sum_{j=v+1}^{u-1}\sigma_{i}\mu_{v,m}q_{i,j}\left(\sigma_{j}\vec{e}_{v}-\sigma_{v}\vec{e}_{j}\right)\\
 & - & \sum_{i=v+1}^{u-1}mr_{i}\left(\sigma_{i}\vec{e}_{v}-\sigma_{v}\vec{e}_{i}\right)
\end{eqnarray*}
we can assume that $\bar{\bar{b}}_{u,v}$ is a polynomial which depends
only on $x_{1},\ldots,x_{v}$, without changing the assumption that
$\bar{\bar{b}}_{u,w}=0$ for $w<v$. But we saw that in this situation
Equation \ref{eq:condition--} yields that $\bar{\bar{b}}_{u,v}\in\sum_{i=v+1}^{u-1}\sigma_{i}R_{u-1}$,
so we can actually assume that $\bar{\bar{b}}_{u,v}=0$, as required.
\end{proof}
This finishes the proof of Proposition \ref{prop:step1}. Here is
the second step of the technical lemma's proof:
\begin{prop}
Let $\alpha=I_{n}+A\in\mathbb{\tilde{\mathbb{J}}}_{m,u-1}$ such that
for every $v<u$, $\bar{a}_{u,v}\in\sigma_{u}^{2}\bar{H}_{m}$. Then,
there exists $\delta\in IA^{m}\cap\mathbb{\mathbb{\tilde{J}}}_{m,u-1}$
such that for every $v\neq u$, the $\left(u,v\right)$-th entry of
$\overline{\alpha\delta^{-1}}$ belongs to $\sigma_{u}^{2}\bar{H}_{m}$.
\end{prop}

\begin{proof}
So let $\alpha=I_{n}+A\in\mathbb{\tilde{\mathbb{J}}}_{m,u-1}$ such
that for every $v<u$, , $\bar{a}_{u,v}\in\sigma_{u}^{2}\bar{H}_{m}$.
We remined that by Equation \ref{eq:reminder}, for every $v>u$ we
have $\bar{a}_{u,v}\in\sigma_{u}\left(\sum_{r=1}^{u-1}\mathfrak{\bar{A}}\sigma_{r}\bar{U}_{r,m}+\sigma_{u}^{2}\bar{U}_{u,m}+\mathfrak{\bar{A}}\bar{O}_{m}+\bar{O}_{m}^{2}\right)$.
Hence, we can write explicitly
\[
\bar{a}_{u,v}=\sigma_{u}\left(\sum_{r=1}^{u-1}\sum_{i=1}^{u}\sigma_{i}\sigma_{r}\mu_{r,m}p_{r,i}+\sigma_{u}^{2}\mu_{u,m}q+\sum_{i=1}^{u}m\sigma_{i}r_{i}+m^{2}s\right)
\]
for some $p_{r,i},\,q,\,r_{i},\,s\in R_{u}$. Clearly, as $\mathfrak{\bar{A}}\bar{O}_{m}\supseteq\mathfrak{\bar{A}}\bar{O}_{m}^{2}$,
by dividing $s$ by $\sigma_{i}$ for $1\leq i\leq u$ (with residue),
we can assume that $s\in\mathbb{Z}$. Consider now the following element:
\[
IA^{m}\ni\left(I_{n}+\sigma_{v}E_{u,u}-\sigma_{u}E_{u,v}\right)^{m^{2}}=I_{n}+\sigma_{v}\mu_{v,m^{2}}E_{u,u}-\sigma_{u}\mu_{v,m^{2}}E_{u,v}=\delta'.
\]
By the computation in the proof of Proposition \ref{prop:reduction1},
we obtain that
\[
\mu_{v,m^{2}}\in\sigma_{v}^{2}U_{v,m}+\sigma_{v}O_{m}+O_{m}^{2}
\]
and thus (we remind that $v>u$)
\begin{eqnarray*}
\sigma_{v}\mu_{v,m^{2}} & \in & \sigma_{v}\left(\sigma_{v}^{2}U_{v,m}+\sigma_{v}O_{m}+O_{m}^{2}\right)\subseteq\tilde{J}_{m,u-1,u}\\
\sigma_{u}\mu_{v,m^{2}} & \in & \sigma_{u}\left(\sigma_{v}^{2}U_{v,m}+\sigma_{v}O_{m}+O_{m}^{2}\right)\subseteq\tilde{J}_{m,u-1,v}.
\end{eqnarray*}
In addition, the determinant of $\delta'$ is $x_{v}^{m^{2}}$. Therefore,
$\delta'\in\mathbb{\mathbb{\tilde{J}}}_{m,u-1}$. Observe now that
as $v>u$, under the projection $\sigma_{u+1},\ldots,\sigma_{n}\mapsto0$,
$x_{v}\mapsto1$, and $\delta$ is therefore maped to
\[
\bar{\delta}'=I_{n}-m^{2}\sigma_{u}E_{u,v}.
\]
Thus, if we multiply $\alpha$ from the right by $\delta'^{s}$ we
obtain that the value of the entries in the $u$-th row under the
projection $\sigma_{u+1},\ldots,\sigma_{n}\mapsto0$ does not change,
besides the value of the entry in the $v$-th colmun, which changes
to (see Equation \ref{eq:reminder} for the ideal which contains $\bar{a}_{u,u}$)
\begin{eqnarray*}
\bar{a}_{u,v}-sm^{2}\sigma_{u}\left(1+\bar{a}_{u,u}\right) & \in & \sigma_{u}\left(\sum_{r=1}^{u-1}\mathfrak{\bar{A}}\sigma_{r}\bar{U}_{r,m}+\sigma_{u}^{2}\bar{U}_{u,m}+\mathfrak{\bar{A}}\bar{O}_{m}\right)\\
 &  & +\,\sigma_{u}^{2}\left(\sum_{r=1}^{u}\mathfrak{\bar{A}}\sigma_{r}\bar{U}_{r,m}+\bar{\mathfrak{A}}\bar{O}_{m}+\bar{O}_{m}^{2}\right)\\
 & = & \sigma_{u}\left(\sum_{r=1}^{u-1}\mathfrak{\bar{A}}\sigma_{r}\bar{U}_{r,m}+\sigma_{u}^{2}\bar{U}_{u,m}+\mathfrak{\bar{A}}\bar{O}_{m}\right).
\end{eqnarray*}
Hence, we can assume that $\bar{a}_{u,v}\in\sigma_{u}\sum_{i=1}^{u-1}\sigma_{i}f_{i}+\sigma_{u}^{2}\left(\sum_{r=1}^{u}\sigma_{r}\bar{U}_{r,m}+\bar{O}_{m}\right)=\sigma_{u}\sum_{i=1}^{u-1}\sigma_{i}f_{i}+\sigma_{u}^{2}\bar{H}_{m}$,
for some $f_{i}\in\sum_{r=1}^{u-1}\sigma_{r}\bar{U}_{r,m}+\bar{O}_{m}$.
Define now (the coefficient of $\vec{e}_{v}$ is the value of the
$\left(u,v\right)$-th entry)
\[
\delta_{v}=\left(\begin{array}{ccccc}
 & I_{u-1} &  & 0 & 0\\
-\sigma_{v}\sigma_{u}f_{1} & \cdots & -\sigma_{v}\sigma_{u}f_{u-1} & 1 & \left(\sigma_{u}\sum_{i=1}^{u-1}\sigma_{i}f_{i}\right)\vec{e}_{v}\\
 & 0 &  & 0 & I_{n-u}
\end{array}\right)\in\mathbb{\mathbb{\tilde{J}}}_{m,u-1}.
\]
By proposition \ref{prop:type 1.1}, we obviously have $\delta_{v}\in IA^{m}$.
In addition, as $v>u$, under the projection $\sigma_{u+1},\ldots,\sigma_{n}\mapsto0$
we have
\[
\bar{\delta}_{v}=\left(\begin{array}{ccc}
I_{u-1} & 0 & 0\\
0 & 1 & \sigma_{u}\left(\sum_{i=1}^{u-1}\sigma_{i}f_{i}\right)\vec{e}_{v}\\
0 & 0 & I_{n-u}
\end{array}\right).
\]

Thus, by multiplying $\alpha$ from the right by $\bar{\delta}_{v}^{-1}$
we obtain that the value of the entries in the $u$-th row under the
projection $\sigma_{u+1},\ldots,\sigma_{n}\mapsto0$ does not change,
besides the value of the entry in the $v$-th colmun, which changes
to
\begin{eqnarray*}
\bar{a}_{u,v}-\sigma_{u}\left(\sum_{i=1}^{u-1}\sigma_{i}f_{i}\right)\left(1+\bar{a}_{u,u}\right) & \in & \sigma_{u}^{2}\bar{H}_{m}+\sigma_{u}^{2}\left(\sum_{r=1}^{u}\mathfrak{\bar{A}}\sigma_{r}\bar{U}_{r,m}+\mathfrak{A}\bar{O}_{m}+\bar{O}_{m}^{2}\right)\\
 & = & \sigma_{u}^{2}\bar{H}_{m}.
\end{eqnarray*}
Thus, defininig $\delta=\prod_{v=u+1}^{n}\delta_{v}$ finishes the
proof of the proposition, and hence, also the proof of the technical
lemma.
\end{proof}

\subsection{\label{subsec:Finishing}Finishing Lemma \ref{thm:stage 1}'s proof}

We remind that we fixed a constant $1\leq u\leq n$. We remind also
that by Corollary \ref{cor:reduction}, it suffices to show that given
$\alpha\in\mathbb{\tilde{\mathbb{J}}}_{m,u-1}$ there exist $\beta\in IA^{m}\cap\mathbb{J}_{m,u}$
and $\gamma\in ISL_{n-1,u}\left(\sigma_{u}H_{m}\right)\cap\mathbb{\mathbb{J}}_{m,u}$
such that $\gamma\alpha\beta\in\tilde{\mathbb{A}}{}_{u}$. 

So let $\alpha=I_{n}+A\in\mathbb{\mathbb{\tilde{J}}}_{m,u-1}$. By
the above technical lemma, there exists $\delta\in IA^{m}\cap\mathbb{\mathbb{\tilde{J}}}_{m,u-1}\subseteq IA^{m}\cap\mathbb{J}_{m,u}$
such that for every $v\neq u$, the $\left(u,v\right)$-th entry of
$\overline{\alpha\delta^{-1}}$ belongs to $\sigma_{u}^{2}\bar{H}_{m}$.
Thus, by replacing $\alpha$ with $\alpha\delta^{-1}$, with out loss
of generality, we can assume that we have $\bar{a}_{u,v}\in\sigma_{u}^{2}\bar{H}_{m}$
for every $v\neq u$. I.e. for every $v\neq u$ one can write $\bar{a}_{u,v}=\sigma_{u}^{2}\bar{b}_{u,v}$
for some $\bar{b}_{u,v}\in\bar{H}_{m}$. 

Now, for every $v\neq u$ define the matrix
\[
\delta_{v}=I_{n}+\left(\begin{array}{c}
\sigma_{1}\bar{b}_{u,v}\left(\sigma_{v}\vec{e}_{u}-\sigma_{u}\vec{e}_{v}\right)\\
\sigma_{2}\bar{b}_{u,v}\left(\sigma_{v}\vec{e}_{u}-\sigma_{u}\vec{e}_{v}\right)\\
\vdots\\
\sigma_{n}\bar{b}_{u,v}\left(\sigma_{v}\vec{e}_{u}-\sigma_{u}\vec{e}_{v}\right)
\end{array}\right)\in\mathbb{J}_{m,u}
\]
which is equals, by direct computation, to the product of the matrices
\[
\mathbb{J}_{m,u}\ni\varepsilon_{v,k}=I_{n}+\left(\begin{array}{c}
0\\
\sigma_{k}\bar{b}_{u,v}\left(\sigma_{v}\vec{e}_{u}-\sigma_{u}\vec{e}_{v}\right)\\
0
\end{array}\right)\leftarrow k\textrm{-th\,\,\ row}
\]
for $k\neq u,v$ and the matrix (the following is an example for $v>u$)
\[
\mathbb{J}_{m,u}\ni\eta_{v}=I_{n}+\left(\begin{array}{c}
0\\
\sigma_{u}\bar{b}_{u,v}\left(\sigma_{v}\vec{e}_{u}-\sigma_{u}\vec{e}_{v}\right)\\
0\\
\sigma_{v}\bar{b}_{u,v}\left(\sigma_{v}\vec{e}_{u}-\sigma_{u}\vec{e}_{v}\right)\\
0
\end{array}\right)\begin{array}{c}
\leftarrow u\textrm{-th\,\,\ row}\\
\\
\leftarrow v\textrm{-th\,\,\ row}
\end{array}
\]
i.e. $\delta_{v}=\eta_{v}\cdot\prod_{u,v\neq k=1}^{n}\varepsilon_{v,k}$
(observe that the matrices $\varepsilon_{v,k}$ commute, so the product
is well defined). One can see that by Propositions \ref{prop:type 1.1}
and \ref{prop:type 1.2}, $\varepsilon_{v,k}\in IA^{m}$ for every
$k\neq u,v$. Moreover, by Proposition \ref{prop:type 2}, $\eta_{v}\in IA^{m}$.
Hence, $\delta_{v}\in IA^{m}\cap\mathbb{J}_{m,u}$. Now, as for every
$1\leq i\leq n$ we have $\sum_{j=1}^{n}a_{i,j}\sigma_{j}=0$ (by
the condition $A\vec{\sigma}=\vec{0}$), $\alpha\cdot\prod_{u\neq v=1}^{n}\delta_{v}$
is equals to
\[
\left[I_{n}+\left(\begin{array}{ccc}
a_{1,1} & \cdots & a_{1,n}\\
\vdots &  & \vdots\\
a_{n,1} & \cdots & a_{n,n}
\end{array}\right)\right]\prod_{u\neq v=1}^{n}\left[I_{n}+\left(\begin{array}{c}
\sigma_{1}\bar{b}_{u,v}\left(\sigma_{v}\vec{e}_{u}-\sigma_{u}\vec{e}_{v}\right)\\
\sigma_{2}\bar{b}_{u,v}\left(\sigma_{v}\vec{e}_{u}-\sigma_{u}\vec{e}_{v}\right)\\
\vdots\\
\sigma_{n}\bar{b}_{u,v}\left(\sigma_{v}\vec{e}_{u}-\sigma_{u}\vec{e}_{v}\right)
\end{array}\right)\right]
\]
\[
=I_{n}+\left(\begin{array}{ccc}
a_{1,1} & \cdots & a_{1,n}\\
\vdots &  & \vdots\\
a_{n,1} & \cdots & a_{n,n}
\end{array}\right)+\sum_{u\neq v=1}^{n}\left(\begin{array}{c}
\sigma_{1}\bar{b}_{u,v}\left(\sigma_{v}\vec{e}_{u}-\sigma_{u}\vec{e}_{v}\right)\\
\sigma_{2}\bar{b}_{u,v}\left(\sigma_{v}\vec{e}_{u}-\sigma_{u}\vec{e}_{v}\right)\\
\vdots\\
\sigma_{n}\bar{b}_{u,v}\left(\sigma_{v}\vec{e}_{u}-\sigma_{u}\vec{e}_{v}\right)
\end{array}\right).
\]
It is easy to see now that if we denote $\alpha\cdot\prod_{u\neq v=1}^{n}\delta_{v}=I_{n}+C$,
then for every $v\neq u$, $\bar{c}_{u,v}=0$, when $c_{i,j}$ is
the $\left(i,j\right)$-th entry of $C$. Hence, we also have
\[
\bar{c}_{u,u}\sigma_{u}=\sum_{v=1}^{n}\bar{c}_{u,v}\bar{\sigma}_{v}=0\,\,\,\,\Longrightarrow\,\,\,\,\bar{c}_{u,u}=0.
\]

Thus, we can write $\overline{\alpha\cdot\prod_{u\neq v=1}^{n}\delta_{v}}=I_{n}+\bar{C}$
when the matrix $\bar{C}$ has the following properties:
\begin{itemize}
\item The entries of the $u$-th row of $\bar{C}$ are all $0$.
\item As $a_{i,v}\in\tilde{J}_{m,u-1,v}$ for every $i,v$, by the computation
for Equation \ref{eq:reminder} we have $\bar{a}_{i,v}\in\sigma_{u}\left(\sum_{r=1}^{u-1}\mathfrak{\bar{A}}\sigma_{r}\bar{U}_{r,m}+\sigma_{u}^{2}\bar{U}_{u,m}+\mathfrak{\bar{A}}\bar{O}_{m}+\bar{O}_{m}^{2}\right)$
for every $i,v\neq u$. Hence, for every $i,v\neq u$ we have
\begin{eqnarray*}
\bar{c}_{i,v} & \in & \sigma_{u}\left(\sum_{r=1}^{u-1}\mathfrak{\bar{A}}\sigma_{r}\bar{U}_{r,m}+\sigma_{u}^{2}\bar{U}_{u,m}+\mathfrak{\bar{A}}\bar{O}_{m}+\bar{O}_{m}^{2}\right)+\sigma_{u}\mathfrak{\bar{A}}\bar{H}_{m}\\
 & = & \sigma_{u}\left(\mathfrak{\bar{A}}\bar{H}_{m}+\bar{O}_{m}^{2}\right).
\end{eqnarray*}
\end{itemize}
Now, as $\det(\delta_{v})=1$ for every $v\neq u$, $\det(\overline{\alpha\cdot\prod_{u\neq v=1}^{n}\delta_{v}})=\det(\overline{\alpha})=\prod_{i=1}^{u}x_{i}^{s_{i}m^{2}}$.
However, as the entries of $\bar{C}$ have the above properties, this
determinant is mapped to $1$ under the projection $\sigma_{u}\mapsto0$.
Thus, $\det(\overline{\alpha\cdot\prod_{u\neq v=1}^{n}\delta_{v}})$
is of the form $x_{u}^{s_{u}m^{2}}$. Now, set $i_{0}\neq u$, and
denote
\[
\zeta=I_{n}+\sigma_{u}\mu_{u,m^{2}}E_{i_{0},i_{0}}-\sigma_{i_{0}}\mu_{u,m^{2}}E_{i_{0},u}=\left(I_{n}+\sigma_{u}E_{i_{0},i_{0}}-\sigma_{i_{0}}E_{i_{0},u}\right)^{m^{2}}\in IA^{m}.
\]
By the computation in the proof of Proposition \ref{prop:reduction1},
we obtain that
\[
\mu_{u,m^{2}}\in\sigma_{u}^{2}U_{u,m}+\sigma_{u}O_{m}+O_{m}^{2}
\]
and thus
\begin{eqnarray*}
\sigma_{u}\mu_{u,m^{2}} & \in & \sigma_{u}\left(\sigma_{u}^{2}U_{u,m}+\sigma_{u}O_{m}+O_{m}^{2}\right)\subseteq\sigma_{u}\left(\mathfrak{\bar{A}}\bar{H}_{m}+\bar{O}_{m}^{2}\right)\subseteq J_{m,u,i_{0}}\\
\sigma_{i_{0}}\mu_{u,m^{2}} & \in & \sigma_{i_{0}}\left(\sigma_{u}^{2}U_{u,m}+\sigma_{u}O_{m}+O_{m}^{2}\right)\subseteq J_{m,u,u}
\end{eqnarray*}
so $\zeta\in IA^{m}\cap\mathbb{J}_{m,u}$. In addition $\det\left(\zeta\right)=x_{u}^{m^{2}}$.
Therefore, $\overline{\alpha\cdot\prod_{v\neq u}\delta_{v}\zeta^{-s_{u}}}$,
writen as $I_{n}+\bar{C}$, has the following properties: 
\begin{itemize}
\item The entries of the $u$-th row of $\bar{C}$ are all $0$.
\item For every $i,v\neq u$ we have $\bar{c}_{i,v}\in\sigma_{u}\left(\mathfrak{\bar{A}}\bar{H}_{m}+\bar{O}_{m}^{2}\right)$,
so we can write $\bar{c}_{i,v}=\sigma_{u}d_{i,v}$ for some $d_{i,v}\in\mathfrak{\bar{A}}\bar{H}_{m}+\bar{O}_{m}^{2}$.
\item For every $1\leq i\leq n$ we have $\sum_{k=1}^{u}\sigma_{k}\bar{c}_{i,k}=0$,
so $\bar{c}_{i,u}=-\sum_{k=1}^{u-1}\sigma_{k}d_{i,k}$ .
\item $\det\left(I_{n}+\bar{C}\right)=1$.
\end{itemize}
In other words
\[
\bar{c}_{i,j}=\begin{cases}
0 & i=u\\
-\sum_{k=1}^{u-1}\sigma_{k}d_{i,k} & j=u\\
\sigma_{u}d_{i,j} & i,j\neq u
\end{cases}
\]
for some $d_{i,j}\in\mathfrak{\bar{A}}\bar{H}_{m}+\bar{O}_{m}^{2}$,
and $\det\left(I_{n}+\bar{C}\right)=1$.

Define now $\beta=\prod_{v\neq u}\delta_{v}\zeta^{-s_{u}}$, so $\beta\in IA^{m}\cap\mathbb{J}_{m,u}$.
In addition, define $\gamma$ to be the inverse of $\gamma^{-1}=I_{n}+\tilde{C}$
where
\[
\tilde{c}{}_{i,j}=\begin{cases}
0 & i=u\\
-\sum_{u\neq k=1}^{n}\sigma_{k}d_{i,k} & j=u\\
\sigma_{u}d_{i,j} & i,j\neq u.
\end{cases}
\]
is the $\left(i,j\right)$-th entry of $\tilde{C}$. Notice that $\gamma^{-1}\in IA\left(\Phi_{n}\right)$,
and that $\overline{\gamma^{-1}}=I_{n}+\bar{C}=\overline{\alpha\beta}$.
In addition
\[
\det(\gamma^{-1})=\det(I_{n}+\tilde{C})=\det(I_{n}+\bar{C})=1.
\]
Moreover, as $d_{i,j}\in\mathfrak{\bar{A}}\bar{H}_{m}+\bar{O}_{m}^{2}\subseteq H_{m}$,
$\gamma\in ISL_{n-1,u}\left(\sigma_{u}H_{m}\right)$. Additionally,
$\gamma\in\mathbb{J}_{m,u}$. Hence, we obtained $\beta\in IA^{m}\cap\mathbb{J}_{m,u}$
and $\gamma\in ISL_{n-1,u}\left(\sigma_{u}H_{m}\right)\cap\mathbb{J}_{m,u}$
such that $\overline{\gamma\alpha\beta}=I_{n}$, i.e. $\gamma\alpha\beta\in\tilde{\mathbb{A}}{}_{u}$,
as required.

\section{\label{sec:Inferences}Remarks and problems for further research}

We will prove now Theorem \ref{thm:full}, which asserts that $C\left(\Phi_{n}\right)$
is abelian for every $n\geq4$. But before, let us state the following
proposition, which is slightly more general than Lemma 2.1. in \cite{key-5},
but proven by similar arguments:
\begin{prop}
\label{prop:exact-1}Let $1\to G_{1}\overset{\alpha}{\to}G_{2}\overset{\beta}{\to}G_{3}\to1$
be a short exact sequence of groups. Assume also that $G_{1}$ is
finitely generated. Then:

1. The sequence $\hat{G}_{1}\overset{\hat{\alpha}}{\to}\hat{G}_{2}\overset{\hat{\beta}}{\to}\hat{G}_{3}\to1$
is also exact.

2. The kernel $\ker(\hat{G}_{1}\overset{\hat{\alpha}}{\to}\hat{G}_{2})$
is central in $\hat{G}_{1}$.
\end{prop}

\begin{proof}
(of Theorem \ref{thm:full}) By Proposition \ref{prop:exact-1}, the
commutative exact diagram
\[
\begin{array}{ccccccccc}
1 & \to & IA\left(\Phi_{n}\right) & \to & Aut\left(\Phi_{n}\right) & \to & GL_{n}\left(\mathbb{Z}\right) & \to & 1\\
 &  &  & \searrow & \downarrow &  & \downarrow\\
 &  &  &  & Aut(\hat{\Phi}_{n}) & \to & GL_{n}(\hat{\mathbb{Z}}) & .
\end{array}
\]
gives rise to the commutative exact diagram
\[
\begin{array}{ccccccc}
\widehat{IA\left(\Phi_{n}\right)} & \to & \widehat{Aut\left(\Phi_{n}\right)} & \to & \widehat{GL_{n}\left(\mathbb{Z}\right)} & \to & 1\\
 & \searrow & \downarrow &  & \downarrow\\
 &  & Aut(\hat{\Phi}_{n}) & \to & GL_{n}(\hat{\mathbb{Z}})
\end{array}
\]
Now, as $n\geq4$, by the CSP for $GL_{n}\left(\mathbb{Z}\right)$,
the map $\widehat{GL_{n}\left(\mathbb{Z}\right)}\to GL_{n}(\hat{\mathbb{Z}})$
is injective, so one obtains by diagram chasing, that $C\left(IA\left(\Phi_{n}\right),\Phi_{n}\right)=\ker(\widehat{IA\left(\Phi_{n}\right)}\to Aut(\hat{\Phi}_{n}))$
is mapped onto $C\left(\Phi_{n}\right)=\ker(\widehat{Aut\left(\Phi_{n}\right)}\to Aut(\hat{\Phi}_{n}))$
through the map $\widehat{IA\left(\Phi_{n}\right)}\to\widehat{Aut\left(\Phi_{n}\right)}$.
In particular, as by Theorem \ref{thm:main} $C\left(IA\left(\Phi_{n}\right),\Phi_{n}\right)$
is central in $\widehat{IA\left(\Phi_{n}\right)}$ for every $n\geq4$,
it is also abelian, and thus $C\left(\Phi_{n}\right)$ is an image
of an abelian group, and therfore abelian, as required. 
\end{proof}
\begin{problem}
\label{prob:Is1}Is $C\left(\Phi_{n}\right)$ not finitely generated?
trivial? 
\end{problem}

We proved in \cite{key-14} that $C\left(IA\left(\Phi_{n}\right),\Phi_{n}\right)$
is not finitely generated for every $n\geq4$. This may suggest that
also $C\left(\Phi_{n}\right)$ is not finitely generated, or at least,
not trivial. Moreover, if $C\left(IA\left(\Phi_{n}\right),\Phi_{n}\right)$
were not central in $\widehat{IA\left(\Phi_{n}\right)}$, we could
use the fact that $IA\left(\Phi_{n}\right)$ is finitely generated
for every $n\geq4$ \cite{key-24}, and by the second part of Proposition
\ref{prop:exact-1} we could derive that the image of $C\left(IA\left(\Phi_{n}\right),\Phi_{n}\right)$
in $\widehat{Aut\left(\Phi_{n}\right)}$ is not trivial. However,
we showed that $C\left(IA\left(\Phi_{n}\right),\Phi_{n}\right)$ is
central in $\widehat{IA\left(\Phi_{n}\right)}$, so it is possible
that $C\left(IA\left(\Phi_{n}\right),\Phi_{n}\right)\subseteq\ker(\widehat{IA\left(\Phi_{n}\right)}\to\widehat{Aut\left(\Phi_{n}\right)})$
and thus $C\left(\Phi_{n}\right)$ is trivial.

We saw in \cite{key-14} that for every $i$ there is a natural surjective
map 
\[
\hat{\rho}_{i}:\widehat{IA\left(\Phi_{n}\right)}\twoheadrightarrow\widehat{GL_{n-1}\left(\mathbb{Z}[x_{i}^{\pm1}],\sigma_{i}\mathbb{Z}[x_{i}^{\pm1}]\right)}.
\]
These maps enabled us to show in \cite{key-14} that for every $n\geq4$,
$C\left(IA\left(\Phi_{n}\right),\Phi_{n}\right)$ can be written as
\[
C\left(IA\left(\Phi_{n}\right),\Phi_{n}\right)=(C\left(IA\left(\Phi_{n}\right),\Phi_{n}\right)\cap_{i=1}^{n}\ker\hat{\rho}_{i})\rtimes\prod_{i=1}^{n}C_{i}
\]
where
\begin{eqnarray*}
C_{i} & \cong & \ker(\widehat{GL_{n-1}\left(\mathbb{Z}[x_{i}^{\pm1}],\sigma_{i}\mathbb{Z}[x_{i}^{\pm1}]\right)}\to GL_{n-1}(\widehat{\mathbb{Z}[x_{i}^{\pm1}]}))\\
 & \cong & \ker(\widehat{SL_{n-1}\left(\mathbb{Z}[x_{i}^{\pm1}]\right)}\to SL_{n-1}(\widehat{\mathbb{Z}[x_{i}^{\pm1}]})).
\end{eqnarray*}
are central in $\widehat{IA\left(\Phi_{n}\right)}$. Here we showed
that also $C\left(IA\left(\Phi_{n}\right),\Phi_{n}\right)\cap_{i=1}^{n}\ker\hat{\rho}_{i}$
lie in the center of $\widehat{IA\left(\Phi_{n}\right)}$ but we still
do not know to determine whether:
\begin{problem}
\label{prob:Is2}Is $C\left(IA\left(\Phi_{n}\right),\Phi_{n}\right)=\prod_{i=1}^{n}C_{i}$
or does it contain more elements?
\end{problem}

It seems that having the answer to Problem \ref{prob:Is2} will help
to solve Problem \ref{prob:Is1}.

\section{Index of notations}
\begin{itemize}
\item $F_{n}$ = the free group on $n$ elements, Section \ref{sec:structure}.
\item $\Phi_{n}=F_{n}/F''_{n}$= the free metabelian group on $n$ elements,
Section \ref{sec:structure}.
\item $\Phi_{n,m}=\Phi_{n}/M_{n,m}$, where $M_{n,m}=\left(\Phi'_{n}\Phi_{n}^{m}\right)'\left(\Phi'_{n}\Phi_{n}^{m}\right)^{m}$,
Section \ref{sec:structure}.
\item $IA(\Phi_{n})=\ker\left(Aut\left(\Phi_{n}\right)\to Aut\left(\Phi_{n}/\Phi_{n}'\right)\right)$,
Section \ref{sec:structure}.
\item $IG_{m}=IG_{n,m}=G(M_{n,m})=\ker\left(IA\left(\Phi_{n}\right)\to Aut\left(\Phi_{n,m}\right)\right)$,
Section \ref{sec:structure}.
\item $IA^{m}=IA_{n}^{m}=\left\langle IA\left(\Phi_{n}\right)^{m}\right\rangle $,
Section \ref{sec:structure}.
\item $IA_{n,m}=\cap\left\{ N\vartriangleleft IA\left(\Phi_{n}\right)\,|\,[IA\left(\Phi_{n}\right):N]\,|\,m\right\} $,
Section \ref{sec:structure-1}.
\item $R=R_{n}=\mathbb{Z}[x_{1}^{\pm1},\ldots,x_{n}^{\pm1}]$ where $x_{1},\ldots,x_{n}$
are free commutative variables, Section \ref{sec:structure}.
\item $\mathbb{Z}_{m}=\mathbb{Z}/m\mathbb{Z}$, Section \ref{sec:structure}.
\item $\sigma_{i}=x_{i}-1$ for $1\leq i\leq n$, Section \ref{sec:structure}.
\item $\vec{\sigma}$ = the column vector which has $\sigma_{i}$ in its
$i$-th entry, Section \ref{sec:structure}.
\item $\mu_{r,m}=\sum_{i=0}^{m-1}x_{r}^{i}$, Section \ref{sec:elementary}.
\item $\mathfrak{A}=\mathfrak{A}_{n}=\sum_{i=1}^{n}\sigma_{i}R_{n}\vartriangleleft R_{n}$
= the augmentation ideal of $R_{n}$, Section \ref{sec:structure}.
\item $\mathfrak{\bar{A}}=\mathfrak{A}_{u}=\sum_{i=1}^{u}\sigma_{i}R_{u}\vartriangleleft R_{u}$,
where $1\leq u\leq n$, Subsection \ref{subsec:5.2}.
\item $\bar{\bar{\mathfrak{A}}}=\mathfrak{A}_{u-1}=\sum_{i=1}^{u-1}\sigma_{i}R_{u-1}\vartriangleleft R_{u-1}$,
where $1\leq u\leq n$, Subsection \ref{subsec:5.2}.
\item $\mathfrak{\tilde{A}}_{u}=\sum_{r=u+1}^{n}\sigma_{r}R_{n}\vartriangleleft R_{n}$,
where $0\leq u\leq n$, Subsection \ref{subsec:5.1}.
\item $\tilde{\mathbb{A}}_{u}=IA\left(\Phi_{n}\right)\cap GL_{n}(R,\mathfrak{\tilde{A}}_{u})$,
where $0\leq u\leq n$, Subsection \ref{subsec:5.1}.
\item $O_{m}=mR_{n}\vartriangleleft R_{n}$, Section \ref{sec:The-main-lemma}.
\item $\bar{O}_{m}=mR_{u}\vartriangleleft R_{u}$, where $1\leq u\leq n$,
Subsection \ref{subsec:5.2}.
\item $\bar{\bar{O}}_{m}=mR_{u-1}\vartriangleleft R_{u-1}$, where $1\leq u\leq n$,
Subsection \ref{subsec:5.2}.
\item $U_{r,m}=\mu_{r,m}R_{n}\vartriangleleft R_{n}$, Section \ref{sec:The-main-lemma}.
\item $\bar{U}_{r,m}=\mu_{r,m}R_{u}\vartriangleleft R_{u}$, where $1\leq u\leq n$,
Subsection \ref{subsec:5.2}.
\item $\bar{\bar{U}}_{r,m}=\mu_{r,m}R_{u-1}\vartriangleleft R_{u-1}$, where
$1\leq u\leq n$, Subsection \ref{subsec:5.2}.
\item $H_{m}=H_{n,m}=\sum_{i=1}^{n}\left(x_{i}^{m}-1\right)R_{n}+mR_{n}\vartriangleleft R_{n}$,
Section \ref{sec:structure}.
\item $\bar{H}_{m}=H_{u,m}=\sum_{r=1}^{u}\sigma_{r}\bar{U}_{r,m}+\bar{O}_{m}\vartriangleleft R_{u}$,
where $1\leq u\leq n$, Subsection \ref{subsec:5.2}.
\item $J_{m}=\sum_{r=1}^{n}\sigma_{r}^{3}U_{r,m}+\mathfrak{A}^{2}O_{m}+\mathfrak{A}O_{m}^{2}\vartriangleleft R_{n}$,
Subsection \ref{subsec:5.1}.
\item $\mathbb{\mathbb{J}}_{m}=\left\{ I_{n}+A\,|\,\begin{array}{c}
I_{n}+A\in IA\left(\Phi_{n}\right)\cap GL_{n}\left(R,J_{m}\right)\\
\det\left(I_{n}+A\right)=\prod_{r=1}^{n}x_{r}^{s_{r}m^{2}},\,\,s_{r}\in\mathbb{Z}
\end{array}\right\} $, Subsection \ref{subsec:5.1}.
\item $\tilde{J}_{m,u,v}=\begin{cases}
\mathfrak{\tilde{A}}_{u}\left(\sum_{r=1}^{u}\mathfrak{A}\sigma_{r}U_{r,m}+\mathfrak{A}O_{m}+O_{m}^{2}\right)+\\
\sum_{r=u+1}^{n}\sigma_{r}^{3}U_{r,m} & v\leq u\\
\mathfrak{\tilde{A}}_{u}\left(\sum_{r=1}^{u}\mathfrak{A}\sigma_{r}U_{r,m}+\mathfrak{A}O_{m}+O_{m}^{2}\right)+\\
\sum_{v\neq r=u+1}^{n}\sigma_{r}^{3}U_{r,m}+\mathfrak{A}\sigma_{v}^{2}U_{v,m} & v>u
\end{cases}$, \\
\\
where $0\leq u\leq n$ and $1\leq v\leq n$, Subsection \ref{subsec:5.1}.
\item $\mathbb{\mathbb{\tilde{J}}}_{m,u}=\left\{ I_{n}+A\in IA\left(\Phi_{n}\right)\,|\,\begin{array}{c}
\det\left(I_{n}+A\right)=\prod_{i=1}^{n}x_{i}^{s_{i}m^{2}}\textrm{,\,\,every\,\,entry\,\,in}\\
\textrm{the\,\,}v\textrm{-th\,\,colmun\,\,of\,\,}A\textrm{\,\,belongs\,\,to\,\,}\tilde{J}_{m,u,v}
\end{array}\right\} $, \\
\\
where $0\leq u\leq n$, Subsection \ref{subsec:5.1}.
\item $J_{m,u,v}=\begin{cases}
\mathfrak{A}\left(\sum_{r=1}^{u}\mathfrak{A}\sigma_{r}U_{r,m}+\mathfrak{A}O_{m}+O_{m}^{2}\right)+\\
\sum_{r=u+1}^{n}\sigma_{r}^{3}U_{r,m} & v\leq u\\
\mathfrak{A}\left(\sum_{r=1}^{u}\mathfrak{A}\sigma_{r}U_{r,m}+\mathfrak{A}O_{m}+O_{m}^{2}\right)+\\
\sum_{v\neq r=u+1}^{n}\sigma_{r}^{3}U_{r,m}+\mathfrak{A}\sigma_{v}^{2}U_{v,m} & v>u
\end{cases}$, \\
\\
where $0\leq u\leq n$ and $1\leq v\leq n$, Subsection \ref{subsec:5.1}.
\item $\mathbb{J}_{m,u}=\left\{ I_{n}+A\in IA\left(\Phi_{n}\right)\,|\,\begin{array}{c}
\det\left(I_{n}+A\right)=\prod_{i=1}^{n}x_{i}^{s_{i}m^{2}}\textrm{,\,\,every\,\,entry\,\,in}\\
\textrm{the\,\,}v\textrm{-th\,\,colmun\,\,of\,\,}A\textrm{\,\,belongs\,\,to\,\,}J_{m,u,v}
\end{array}\right\} $, \\
\\
where $0\leq u\leq n$, Subsection \ref{subsec:5.1}.
\item $E_{d}\left(R\right)=\left\langle I_{d}+rE_{i,j}\,|\,r\in R,\,1\leq i\neq j\leq d\right\rangle \leq SL_{d}\left(R\right)$,
where $R$ is a ring and $E_{i,j}$ is the matrix that has $1$ in
its $\left(i,j\right)$-th entry and $0$ elsewhere, Section \ref{sec:structure}.
\item $SL_{d}\left(R,H\right)=\ker\left(SL_{d}\left(R\right)\to SL_{d}\left(R/H\right)\right)$,
where $R$ is a ring and $H\vartriangleleft R$, Section \ref{sec:structure}.
\item $GL_{d}\left(R,H\right)=\ker\left(GL_{d}\left(R\right)\to GL_{d}\left(R/H\right)\right)$,
where $R$ is a ring and $H\vartriangleleft R$, Section \ref{sec:structure}.
\item $E_{d}\left(R,H\right)$ = the normal subgroup of $E_{d}\left(R\right)$,
generated as a normal subgroup by the matrices of the form $I_{d}+hE_{i,j}$
for $h\in H$, Section \ref{sec:structure}.
\item $IGL_{n-1,i}=\left\{ I_{n}+A\in IA\left(\Phi_{n}\right)\,|\,\begin{array}{c}
\textrm{The\,\,}i\textrm{-th\,\, row\,\, of\,\,}A\textrm{\,\, is\,\,0,}\\
I_{n-1}+A_{i,i}\in GL_{n-1}\left(R_{n},\sigma_{i}R_{n}\right)
\end{array}\right\} $, \\
\\
for $1\leq i\leq n$, Section \ref{sec:structure}.
\item $ISL_{n-1,i}\left(H\right)=IGL_{n-1,i}\cap SL_{n-1}\left(R_{n},H\right)$,
under the identification of $IGL_{n-1,i}$ with $GL_{n-1}\left(R_{n},\sigma_{i}R_{n}\right)$,
Section \ref{sec:structure}.
\item $IE_{n-1,i}\left(H\right)=IGL_{n-1,i}\cap E{}_{n-1}\left(R_{n},H\right)$,
under the identification of the group $IGL_{n-1,i}$ with $GL_{n-1}\left(R_{n},\sigma_{i}R_{n}\right)$,
Section \ref{sec:structure}.
\end{itemize}

Institute of Mathematics\\
The Hebrew University\\
Jerusalem, ISRAEL 91904\\
\\
davidel-chai.ben-ezra@mail.huji.ac.il\\


\begin{thebibliography}{Be1}
\bibitem[A]{key-3} M. Asada, The faithfulness of the monodromy representations
associated with certain families of algebraic curves, J. Pure Appl.
Algebra 159 (2001), 123\textendash 147.

\bibitem[Ba]{key-13} S. Bachmuth, Automorphisms of free metabelian
groups, Trans. Amer. Math. Soc. 118 (1965) 93-104.

\bibitem[Be1]{key-6} D. E-C. Ben-Ezra, The congruence subgroup problem
for the free metabelian group on two generators. Groups Geom. Dyn.
10 (2016), 583\textendash 599.

\bibitem[Be2]{key-14} D. E-C. Ben-Ezra, The IA-congruence kernel
of high rank free Metabelian groups, arXiv:1707.09854.

\bibitem[Bi]{key-36} J. S. Birman, Braids, links, and mapping class
groups, Princeton University Press, Princeton, NJ, University of Tokyo
Press, Toyko, 1975.

\bibitem[Bo1]{key-19} M. Boggi, The congruence subgroup property
for the hyperelliptic modular group: the open surface case, Hiroshima
Math. J. 39 (2009), 351\textendash 362.

\bibitem[Bo2]{key-6-1} M. Boggi, A generalized congruence subgroup
property for the hyperelliptic modular group, arXiv:0803.3841v5.

\bibitem[BER]{key-5} K-U. Bux, M. V. Ershov, A. S. Rapinchuk, The
congruence subgroup property for $Aut\left(F_{2}\right)$: a group-theoretic
proof of Asada's theorem, Groups Geom. Dyn. 5 (2011), 327\textendash 353.

\bibitem[BL]{key-7} D. E-C. Ben-Ezra, A. Lubotzky, \foreignlanguage{american}{The
congruence subgroup problem for low rank free and free metabelian
groups, J. Algebra (2017), http://dx.doi.org/10.1016/j.jalgebra.2017.01.001.} 

\bibitem[BLS]{key-23} H. Bass, M. Lazard, J.-P. Serre, Sous-groupes
d'indice fini dans $SL_{n}\left(\mathbb{Z}\right)$, (French) Bull.
Amer. Math. Soc. 70 (1964) 385\textendash 392.

\bibitem[BM1]{key-2} S. Bachmuth, H. Y. Mochizuki, The non-finite
generation of $AUT(G)$, $G$ free metabelian group of rank $3$,
Trans. Amer. Math. Soc. 270 (1982), 697\textendash 700.

\bibitem[BM2]{key-24} S. Bachmuth, H. Y. Mochizuki, $Aut\left(F\right)\to Aut\left(F/F''\right)$
is surjective for free group $F$ of rank $\geq4$, Trans. Amer. Math.
Soc. 292 (1985), 81\textendash 101.

\bibitem[DDH]{key-16-1} S. Diaz, R. Donagi, D. Harbater, Every curve
is a Hurwitz space, Duke Math. J. 59 (1989), 737\textendash 746.

\bibitem[FM]{key-20} B. Farb, D. Margalit, A primer on mapping class
groups, Princeton Mathematical Series, 49. Princeton University Press,
Princeton, NJ, 2012.

\bibitem[KN]{key-24-1}  M. Kassabov, M. Nikolov, Universal lattices
and property tau, Invent. Math. 165 (2006), 209\textendash 224.

\bibitem[L]{key-4} A. Lubotzky, Free quotients and the congruence
kernel of $SL_{2}$, J. Algebra 77 (1982), 411\textendash 418.

\bibitem[Ma]{key-35-1} W. Magnus, On a theorem of Marshall Hall,
Ann. of Math. 40 (1939), 764\textendash 768.

\bibitem[Mc]{key-18} D. B. McReynolds, The congruence subgroup problem
for pure braid groups: Thurston's proof, New York J. Math. 18 (2012),
925\textendash 942.

\bibitem[Mel]{key-17} O. V. Mel\textasciiacute nikov, Congruence
kernel of the group $SL_{2}\left(\mathbb{Z}\right)$, (Russian) Dokl.
Akad. Nauk SSSR 228 (1976), 1034\textendash 1036.

\bibitem[Men]{key-22} J. L. Mennicke, Finite factor groups of the
unimodular group, Ann. of Math. 81 (1965), 31\textendash 37. 

\bibitem[MKS]{key-1} W. Magnus, A. Karrass, D. Solitar, Combinatorial
group theory: Presentations of groups in terms of generators and relations,
Interscience Publishers, New York-London-Sydney, 1966.

\bibitem[NS]{key-17-1} N. Nikolov, D. Segal, Finite index subgroups
in profinite groups, C. R. Math. Acad. Sci. Paris 337 (2003), 303\textendash 308.

\bibitem[Rom]{key-40} N. S. Romanovski\u{\i}, On Shmel\textasciiacute kin
embeddings for abstract and profinite groups, (Russian) Algebra Log.
38 (1999), 598-612, 639-640, translation in Algebra and Logic 38 (1999),
326\textendash 334.

\bibitem[RS]{key-37} V. N. Remeslennikov, V. G. Sokolov, Some properties
of a Magnus embedding, (Russian) Algebra i Logika 9 (1970), 566\textendash 578,
translation in Algebra and Logic 9 (1970), 342\textendash 349.

\bibitem[Su]{key-33} A. A. Suslin, The structure of the special linear
group over rings of polynomials, (Russian) Izv. Akad. Nauk SSSR Ser.
Mat. 41 (1977), 235\textendash 252, 477.\\
\\
\\
\\
\end{thebibliography}
\end{document}